\documentclass[10pt,a4paper]{amsart}
\usepackage{xcolor}
\definecolor{mg}{rgb}  {0.85, 0.,  0.85}

\newcommand{\bk}{\color{black}}

\usepackage[backend=biber,bibencoding=utf8,citestyle=alphabetic,bibstyle=alphabetic,isbn=false,doi=false,url=false,hyperref=true,backref=true,backrefstyle=three,giveninits]{biblatex}
\renewbibmacro*{volume+number+eid}{%
	\textbf{\printfield{volume}}%
	\setunit*{\adddot}
	\setunit*{\addnbspace}
	\printfield{number}%
	\setunit{\addcomma\space}%
	\printfield{eid}}
\renewbibmacro{in:}{%
	\ifentrytype{article}{}{\printtext{\bibstring{in}\intitlepunct}}}
\DeclareFieldFormat[article]{number}{\mkbibparens{#1}}
\usepackage[UKenglish]{babel}
\usepackage[colorlinks]{hyperref}
\usepackage{amsmath,amssymb,amsthm}
\usepackage{mathrsfs}
\usepackage{verbatim}
\usepackage[utf8]{inputenc}
\usepackage[T1]{fontenc}
\usepackage{lmodern}
\usepackage{graphicx}
\usepackage{caption}
\usepackage{subcaption}
\usepackage{wrapfig}
\usepackage{microtype}
\usepackage{bm}
\usepackage{dsfont}
\usepackage[normalem]{ulem} 
\usepackage[retainorgcmds]{IEEEtrantools}
\usepackage[showframe=false]{geometry}
\usepackage[capitalise,nameinlink]{cleveref}
\usepackage{tikz-cd}
\usetikzlibrary{arrows}
\usetikzlibrary{positioning}

\usepackage{csquotes}
\usepackage{mathtools}
\usepackage{soul}
\usepackage{xfrac}
\allowdisplaybreaks
\crefformat{equation}{(#2#1#3)}
\crefrangeformat{equation}{(#3#1#4) to~(#5#2#6)}
\crefmultiformat{equation}{(#2#1#3)}%
{ and~(#2#1#3)}{, (#2#1#3)}{ and~(#2#1#3)}

\newcommand\ie{{\em i.e.}~}

\def\T{\mathbb{T}}

\def\Z{\mathbb{Z}}
\def\C{\mathbb{C}}

\def\U{\mathscr U}
\def\F{\mathscr F}

\def\A{\mathcal A} 

\def\f{\mathbf f}

\def\Tau{\mathcal T}

\def\R{\mathbb{R}}

\def\V{\mathcal V}

\def\la{\langle}
\def\ra{\rangle}

\def\Im{\rm Im}

\def\d{\mathrm d}

\def\({\left(}
\def\[{\left[}
\def\){\right)}
\def\]{\right]}
\def\l|{\left\lvert}
\def\r|{\right\rvert}
\def\lp{\left\lVert}
\def\rp{\right\rVert}

\def\<{\langle}
\def\>{\rangle}

\def\e{\mathrm e}

\DeclareMathOperator{\rank}{Rank}

\DeclareMathOperator{\Trigpol}{Trig\,Pol}

\newcommand{\norm}[1]{\left\lVert#1\right\rVert}
\newcommand{\Norm}[2]{\left\lVert#1\right\rVert_{#2}}

\newtheorem{Theorem}{Theorem}[section]
\newtheorem{Remark}[Theorem]{Remark}
\newtheorem{Lemma}[Theorem]{Lemma}

\newtheorem{Corollary}[Theorem]{Corollary}
\newtheorem{Proposition}[Theorem]{Proposition}

\def\xib{\bm{\xi}}

\crefname{Lemma}{Lemma}{Lemmas}

\newcommand{\bel}{\begin{equation} \label}
\newcommand{\ee}{\end{equation}}
\newcommand{\ba}{\begin{array}}
\newcommand{\ea}{\end{array}}

\begin{document}

\title{Spectral Asymptotics at Thresholds  for a Dirac-type Operator on $\Z^2$}
\author[P.\ Miranda]{Pablo Miranda}\address{Departamento de Matem\'atica y Ciencia de la Computaci\'on, Universidad de Santiago de Chile, Las Sophoras 173. Santiago, Chile.}\email{pablo.miranda.r@usach.cl}
\author[D.\ Parra]{Daniel Parra}
\address{Departamento de Matem\'atica y Ciencia de la Computaci\'on, Universidad de Santiago de Chile, Las Sophoras 173. Santiago, Chile.}
\email{daniel.parra.v@usach.cl}
\author[G.\ Raikov]{Georgi Raikov$^\dagger$}
\address{Facultad de Matemáticas, Pontificia Universidad Católica de Chile, Av. Vicuña Mackenna 4860, Santiago, Chile.}

\begin{abstract}
 In this article, we provide the spectral analysis of a Dirac-type operator on $\Z^2$ by describing the behavior of the spectral shift function associated with a sign--definite trace--class perturbation by a multiplication operator. We prove that it remains bounded outside a single threshold and obtain its main asymptotic term in the unbounded case. Interestingly, we show that the constant in the main asymptotic term encodes the interaction between a flat band and whole non--constant bands. The strategy used is the reduction of the spectral shift function to the eigenvalue counting function of some compact operator which can be studied as a toroidal pseudo--differential operator.
\end{abstract}

\maketitle

\section{Introduction and main results}

We start by giving a streamlined presentation of our operators acting on $\Z^2$.  
First we define a graph structure by setting the set of vertices $\mathcal{V}=\Z^2$ and the set of oriented edges $\mathcal{A}=\{(x,y)\in\Z^2\times\Z^2:y=x\pm \delta_i\}$, where $\{\delta_1,\delta_2\}$ form the canonical basis of $\Z^2$. We note an edge in $\mathcal{A}$ as $\e=(x,y)$ and its transpose by $\overline{\e}=(y,x)$. Given a vertex $x$, we set  $\mathcal{A}_x=\{\e\in \mathcal{A}:\e=(x,y)\}$.

Les us  denote by $X=(\mathcal{V},\mathcal{A})$ this graph structure defined on $\Z^2$, and  consider the vector spaces of \emph{$0-$cochains} $C^0(X)$ and \emph{$1-$cochains} $C^1(X)$ given by:
\begin{equation*}
C^0(X):=\{f:\mathcal{V}\to\C\}\text{ ; }\qquad C^1(X):=\{f: \mathcal{A}\to\C\mid f(\e)=-f(\overline{\e})\}\text{ .}
\end{equation*}
We will denote by $C(X)$ the direct sum $C^{0}(X)\oplus C^{1}(X)$ and  will refer to an $f\in C{(X)}$ as a \emph{cochain}. This suggests the notation $C_c(X)$ for the space of cochains that are finitely supported.

For $f,g\in C(X)$ we take the   inner product 
\begin{equation}\label{interno}
\la f,g\ra:=\sum_{x\in \V}f(x)\overline{g(x)}+\frac12 \sum_{\e\in \A}f(\e)\overline{g(\e)}\ .
\end{equation}
The Hilbert space $l^2(X)$ is defined as the closure of $C_c(X)$ in the norm induced by \cref{interno}. It coincides with $\{f\in C(X)\mid\lp f\rp:=\la f,f\ra^{\frac12}<\infty\}$.

The \emph{coboundary operator} (or \emph{difference operator}) $d:C^0(X)\to C^1(X)$ is defined by:
\begin{equation*}
df(\e):=f(y)-f(x)\text{, for }\e=(x,y)\ .
\end{equation*}
Its formal adjoint $d^*:C^1(X)\to C^0(X)$ is given by the finite sum 
\begin{equation*}
d^*f(x)=-\sum_{\e\in A_x}f(\e)\ .
\end{equation*}
Note that since $A_x=\{(x,x+\delta_1),(x,x+\delta_2),(x,x-\delta_1),(x,x-\delta_2)\}$, the sum has only four terms.

We can now define the  operator  on $l^2(X)$
\begin{equation*}
H_0:=\begin{pmatrix}
m & d^*\\
d & -m
\end{pmatrix}\ ,
\end{equation*}
where $m$ is just the multiplication operator by the constant $m\geq0$. It is not difficult to see that 
$$H_0^2=\begin{pmatrix}
\Delta_{\mathcal V} +m^2& 0\\
0 & \Delta_{\mathcal A}+m^2
\end{pmatrix}\ ,$$
where $\Delta_{\mathcal V}$ and $\Delta_{\mathcal A}$  are the discrete Laplacians in edges and vertices, respectively. So it is natural to say that   $H_0$ is a Dirac-type operator \cite{Ec45,AT15,Pa17}.

The operator $H_0$ is bounded self-adjoint, and  is analytically fibered over $\T^2$ (see \cref{subsec:Decom} for details). The analysis of its band functions shows that the spectrum of $H_0$ is 
\begin{equation}\label{23sep21}
\sigma(H_0)=\sigma_{ac}(H_0) 
=[-\sqrt{m^2+8},-m]\bigcup[m,\sqrt{m^2+8}]\ .
\end{equation}  Moreover, 
 there exist a discrete set $\mathcal{T}$ of thresholds in the spectrum of $H_0$, given by  \begin{equation}\label{eq:tau1}
\mathcal{T}=\left\{\pm m,\pm\sqrt{m^2+4},\pm\sqrt{m^2+8} \right\}\ .
 \end{equation}
In this article we understand the set of thresholds as those point in the spectrum of $H_0$ where the Mourre estimates does not hold and hence we do not have \emph{a priori} a Limiting Absorption Principle (LAP) \cite{Pa17}. 

We will see that there are three types of thresholds and one could expect qualitatively different properties  of the spectrum near each type: for $m>0$   the points  $\{\pm \sqrt{m^2+8},m\}$ are thresholds of \emph{elliptic} type,   $\{\pm\sqrt{m^2+4}\}$ are  of \emph{hyperbolic} type, and   $-m$ is elliptic and an eigenvalue of infinite multiplicity. The case   $m=0$ is particular:  there is no finite gap in the spectrum and  the threshold at zero is not of elliptic type. It is called   a \emph{Dirac point} (see \cref{subsec:Spectral}). 

In this article we will study spectral properties of $H_0\pm V$, where $V\geq0$ is a potential that is defined both on vertices and edges, and satisfies $V(\e)=V(\overline{\e})$.  
Use the notation  \begin{equation}\label{eq:H0V}
H_\pm:=H_0\pm V \ .
\end{equation}
 Since the Dirac delta functions on vertices and edges lie in $l^2(X)$, we have that the range of $V$ coincides with its set of eigenvalues. From this it is clear that the trace class norm of $V$ is
\begin{equation*}
\Norm{V}{1}=\sum_{x\in\mathcal{V}}|V(x)|+\frac12 \sum_{\e\in\mathcal{A}}|V(\e)|\ .
\end{equation*}

When   $V$ is of trace class,   there exists a unique $\xib(\lambda;H_\pm,H_0)$  in $L^1(\R)$ that satisfies the trace formula
$$\text{Tr}(f(H_\pm)-f(H_0))=\int_\R \d \lambda \,\xib(\lambda;H_\pm,H_0) f'(\lambda), $$
for all $f\in C^\infty_0(\R)$ (see the original work  \cite{MR0060742}, or the monograph \cite[Ch. 8]{yaf}).  The function $\xib(\cdot;H_\pm,H_0)$  is called the Spectral Shift Function (SSF) for the pair $(H_\pm,H_0)$. This function can be defined in an abstract setting and  is an important object in the analysis of linear operators. For instance, it  is related to the scattering matrix $S(\lambda; H_\pm,H_0)$ by the Birman-Krein formula \cite{MR0139007,yaf}
$${\rm det} S(\lambda; H_\pm,H_0) = e^{-2\pi i \xib(\lambda;H_\pm,H_0)}, \quad a.e. \,\lambda\in\sigma_{{\rm ac}}(H_0).$$ It also gives the number of discrete eigenvalues  of $H_\pm$ outside the essential spectrum (see \cref{SSFasECF} for more details).

 Our main goal in this article is to describe the SSF near the thresholds in the spectrum of $H_0$. These types of results have been obtained in the literature for different models. In particular for the Laplacian in $\R^d$ perturbed by a decaying electric potential. These results are related to Levinson's theorem (see \cite{Ya10,MR1730513} and references therein). More recently, results for  2D and 3D  magnetic Schr\"odinger and Dirac operators were obtained \cite{MR2057679,MR2763346,MR3771834,MR4087369}. In the discrete case, some trace formulas have been obtained for periodic graphs in settings very close to ours \cite{MR4350511,MR2913620}.

 One of the novelties of this work   is related to  the rich  structure of the spectrum of $H_0$, which is due to the remarkable form of the band functions. In particular, we have the interaction of a flat band with a non-constant band function at the maximal point of the last one, as well as the existence of saddle points which implies the appearance of hyperbolic thresholds. To the  best of our knowledge, this is the first article where these situations are studied.


\subsection{Main results}\label{sec:main}

Let $\mu\in\Z^2$,  and define the edges $e_1=((0,0),(1,0))$, $e_2=((0,0),(0,1))$. Take 
the following real-valued functions on $\Z^2$:
\begin{equation}\label{eq:littlev}
v_1(\mu):=V(\mu) \quad ; \quad v_2(\mu):=V(\mu\e_1) \quad \text{and} \quad v_3(\mu):=V(\mu\e_2)\ ,
\end{equation}
where for  $\e=(x,y)$ we have the natural action
\begin{equation*}
\mu \e=(\mu+x,\mu+y)\ .
\end{equation*}
Whenever is convenient, we will keep the notation $x,y$ for elements of $\Z^2$ seen as a vertices and $\mu,\nu$ for elements of $\Z^2$ seen as elements of the group acting on $X$.

To study the most basic  properties  of the SSF  
 it will suffice to assume that
\begin{equation}\label{3dec19_3}
\sum_{\mu\in\Z^2, \langle \mu \rangle >N} |v_j(\mu)|=O(N^{-2\beta_j}),\quad N\to\infty \ ,
\end{equation}
with $0<\beta_j  \leq 1 $ for $j=1,2,3$, where $\langle \mu \rangle=(1+|\mu|^2)^{1/2}$. For instance, it is easy to see that this condition ensures the existence of the SSF, as it implies that  $V$ is summable. Moreover we have the following theorem: 
\begin{Theorem}\label{teo:acotado}
Let us suppose that the perturbation $V$ is positive and each $v_j$ satisfies \cref{3dec19_3} with $\beta_j>0$. Then, on any compact set $\mathcal{K}\subset \R \backslash \{-m\}$,
\begin{equation*} 
\sup_{\lambda\in\mathcal{K}} \xib(\lambda;H_{\pm},H_0)<\infty\ , 
\end{equation*}
\ie the SSF is bounded away from $-m$.
\end{Theorem}

\begin{Remark}\
This  theorem implies in particular that the SSF $\xib(\lambda;H_{\pm},H_0)$ is bounded at the hyperbolic thresholds, and the study of the SSF at hyperbolic thresholds seems to have not been carried out
before in any model. However, there is a growing interest in the understanding of the
spectral properties of discrete Hamiltonians inside their continuous spectrum, where
these kind of thresholds appear \cite{IJ19,MR4270804}. \end{Remark}

One consequence of \cref{teo:acotado} is that the only possible point of unbounded growth of the SSF is $-m$. In our second main theorem we describe the explicit asymptotic behavior of the SSF at that point, for potentials $V$ that shows a power-like decay. More precisely,
for  $\gamma > 0$ and  {for any multi-index} $\alpha$ we assume 
\bel{31may21} |{\rm D}^\alpha v_j(\mu)|\leq C_\alpha \langle \mu \rangle^{-\gamma-\rho|\alpha|}; \qquad j=2,3\ee
where ${\rm D_{\mu_j}} v(\mu):=v(\mu+\delta_j)-v(\mu)$, and ${\rm D}^\alpha:={\rm D}_{\mu_1}^{\alpha_1}...{\rm D}_{\mu_d}^{\alpha_d}. $
Moreover,  
\bel{14jan21}\lim_{|\mu|\to\infty}|\mu|^{\gamma}
 v_l(\mu)=\Gamma_l;\qquad  l=2,3\ ,
\end{equation}
exist and at least one $\Gamma_j\in \R$  is not equal to $0$. 

We define  the matrix 
$$\Gamma:=\begin{pmatrix}
\Gamma_2 &0  \\[.5em]
 0&\Gamma_3\end{pmatrix}\ ,$$
and  the constant  
\bel{28jan20}\mathcal{C}:=\pi\int_{\T^2}\rm{Tr}\left(\Big(\mathcal{A}(\xi)^*\Gamma\mathcal{A}(\xi)\Big)^{2/\gamma}\right)\d \xi\ ,\ee
where 
$$\mathcal{A}:=
 \begin{pmatrix}
 b(\xi)r(\xi)^{-1/2} &0  \\[.5em]
 a(\xi)r(\xi)^{-1/2}&0\end{pmatrix}\ ,$$
and $a$, $b$,  $r=a^2+b^2$ are given by the analytic fibration of $H_0$ (see \cref{3}).
Notice that in condition \cref{14jan21} we do not ask to $v_2$ and $v_3$ to have the same decaying rate. For instance, if $v_2$ decays faster that $v_3$, and hence $\Gamma_2=0$, the constant $\mathcal{C}$ will not depend on  $v_2$. 

\begin{Theorem}\label{mainTh}
Let $V\geq 0$ and suppose that $v_1$ satisfies \cref{3dec19_3} with $\beta_1>0$, and that $v_2$ and $v_3$ satisfy \cref{31may21,14jan21},  with $\gamma>2$. Then,
 \begin{equation}\label{eq:mainasym}
 \xib(\lambda;H_-,H_0)= \begin{cases}
 -\mathcal{C}|\lambda+m|^{-2/\gamma}(1+o(1)) &\text{ if } \,\lambda \uparrow -m\ , \\
 O(1) &\text{ if }\,\lambda \downarrow  -m\ ,
  \end{cases}
  \end{equation}
and 
 \begin{equation}\label{eq:mainasym2}
 \xib(\lambda;H_+,H_0)= \begin{cases}  O(|\ln(|\lambda+m|)|) &\text{ if } \,\lambda \uparrow -m
\ , \\ \mathcal{C}|\lambda+m|^{-2/\gamma}(1+o(1)) &\text{ if }\,\lambda \downarrow  -m
\ .
  \end{cases}
  \end{equation}
\end{Theorem}

\begin{Remark}
 The asymptotic order ${-2}/{\gamma}$ of the SSF is clearly determined by the perturbation $V$. This may be interpreted as the contribution of the constant band at $-m$. The contribution of the non-constant band is instead encoded in the constant $\mathcal{C}$. This constant contains  an explicit interaction between the perturbation and the whole non-constant band functions (see \cref{3}). As far as we know, this is the first time that such a behavior has been observed.
\end{Remark}

\begin{Remark}
\cref{mainTh}  is proved by a  reduction of the SSF to an eigenvalue counting function, and by using  a general theorem on the spectral asymptotics for  integral operators with toroidal symbol. The statement and proof of this last theorem  appears in section \cref{Raikov}. The result is, to the best of our knowledge, new and may be of independent interest. The proof uses an appropriate Cwikel estimate (\cref{Cwikel-Bir-Sol}).
\end{Remark}

\begin{Remark}\label{SSFasECF} Since $V$ is compact, $\sigma_{\rm  ess}(H)=\sigma_{\rm  ess}(H_0)$. Thus, 
when $m>0$ we have that $(-m,m)$  is a gap in the essential spectrum of $H_\pm$. Then,  for $\lambda\in(-m,m)$  we can consider the function 
$$\mathcal{N}^\pm(\lambda)={\rm Rank} \mathds{1}_{(\lambda,m)}(H_\pm)\ ,$$
with $\mathds{1}_{\Omega}$ being  the characteristic function over the Borel set $\Omega$. Clearly, this function counts the number of discrete eigenvalues of $H_\pm$  on the interval $(-m+\lambda,m)$. By  \cref{teo:acotado} these functions are well defined, as there are no accumulation of eigenvalues at $m$ from below. Moreover, for $\lambda\in(-m,m)$ 
$$\xib(\lambda;H_\pm,H_0)=\pm\mathcal{N}^\pm(\lambda)+O(1)\ ,$$
(see \cite{Pu98}).   Similarly, we can write  analogous equations  for the eigenvalues of $H_\pm$ on the intervals $(-\infty,-\sqrt{m^2+8})$ and   $(\sqrt{m^2+8},\infty)$. In consequence, \cref{teo:acotado} and \cref{mainTh} gives  the asymptotic distribution of the discrete eigenvalues of $H_\pm$ (under conditions \cref{3dec19_3,14jan21}). In order to reduce the length of the article and to maintain the symmetry of the results, we plan to consider   $\mathcal{N}^\pm(\lambda)$ for $\gamma \in (0,2]$ elsewhere. 
\end{Remark}

\subsection{Notations and structure of the article}
Throughout the article we will denote by $C$ different constants that appear which are non-essential to our purposes. Whenever $\pm$ appears it refers to two independent statements that are valid independently. Also for non-negative real functions $f,g$ we note
\begin{equation*}
f \asymp g
\end{equation*}
if there exists two non-negative constants $C_1,C_2$ such that 

\begin{equation*}
C_1 f \leq g \leq C_2 f \ .
\end{equation*}

The rest of the paper is organized as follows. In \cref{sec:unperturbed} we recall the integral representation and describe the spectrum of $H_0$. Then, in \cref{sec:SSFrep}, we obtain an explicit expression for the LAP that will be the base of our further investigations. In \cref{sec:alphaenergy} we study the operators that appear in the SSF representation  given  in \cref{sec:SSFrep}. In \cref{sec:hyper} we prove our results for hyperbolic thresholds. In \cref{Raikov} we obtain some abstract results concerning the eigenvalue asymptotics for a class of integral operators. Finally, in \cref{proofs} we put together these results to obtain the proofs of our main theorems in the parabolic  and Dirac point case.

\section{Analysis of the unperturbed operator  \texorpdfstring{$H_0$}{H0}}\label{sec:unperturbed}

\subsection{Integral decomposition}\label{subsec:Decom}
In this section we construct an unitary operator $\U:l^2(X)\to L^2(\T^2,\C^3)$ in order to write $H_0$ as an analytic fibered operator. We take the convention that $\T^2=\R^2/[0,1]^2$ and recall that $\T^2$ is the dual of $\Z^2$ by $\xi(\mu)=e^{2\pi i\,\xi\cdot\mu}$, for $\xi \in \T^2$ and $\mu\in \Z^2$.  The construction of $\U$ is a particularization of the general construction obtained in \cite{Pa17}. 
We define $\U:C_c(X)\to  L^2(\T^2,\C^3)$ by setting, for $f\in C_c(X)$ and $\xi\in\T^2$,
\begin{equation*}
(\U f)(\xi)=\(\sum_{\mu\in\Z^2}e^{-2\pi i\xi \cdot\mu}f(\mu),\sum_{\mu\in\Z^2}e^{-2\pi i\xi\cdot\mu}f(\mu\e_1),\sum_{\mu\in\Z^2}e^{-2\pi i\xi\cdot\mu}f(\mu e_2)\) .
\end{equation*}
We denote by $\Trigpol(\T^2;\C^{3})$ the subspace of $L^2(\T^2;\C^{3})$ composed by functions $\varphi$ that admit
 \begin{equation*}
 \varphi(\xi)=\sum_{\mu\in\Z^2}e^{2\pi i\xi\cdot\mu}\varphi_\mu\text{ ; with }\mathbf{0}\ne\varphi_\mu\in\C^{3} \text{ for only finitely many }\mu\ .
 \end{equation*} 
This  is a conveniently dense space in $L^2(\T^2;\C^{3})$ and coincides with $\U(C_c(X))$. To write the  adjoint of $\U$ define the  index $\eta((x,y))=y-x$ and
fix an orientation on the graph by setting $\mathcal{A}^+=\{\e\in \mathcal{A} : \eta(e)=\delta_i, i=1,2\}$. Set also  the integer part of an edge by  $[(x,y)]:=x$ .  Therefore,  for $\varphi\in\Trigpol(\T^2;\C^{3})$ 
  \begin{align*}
 (\U^*\varphi)(x)=&\int_{\T^2}\d\xi e^{2\pi i\xi \cdot x}\varphi_1(\xi)\ , \\
 (\U^*\varphi)(\e)=&\int_{\T^2}\d\xi e^{2\pi i\xi \cdot \[\e\] }\varphi_{\imath(\e)}(\xi)\quad{\text{ if }}\e\in \mathcal{A}^+\ , \\
 (\U^*\varphi)(\e)=&-\int_{\T^2}\d\xi e^{2\pi i\xi \cdot (\[\e\]+\eta(\e)) }\varphi_{\imath(\e)}(\xi)\quad{\text{ if }}\e\notin \mathcal{A}^+\ ,
  \end{align*}
 where $\imath(e)=j+1$ if there exist $\mu\in\Z^2$ such that either $\mu \e_j=e$ or $\mu \overline{\e}_j=e$. Then $\U$ extends to a unitary operator, still denoted by $\U$, from $l^2(X)$ to $L^2(\T^2,\C^3)$.

 Note that this definition of $\U$ correspond to the following choice of Fourier transform in $\Z^2$:
 \begin{equation*}
\F^*:l^2(\Z^2)\to L^2(\T^2)\quad ; \quad (\F^* f)(\xi)\equiv\hat{f}(\xi) :=\sum_{\mu\in\Z^2}e^{-2\pi i \xi\cdot\mu}f(\mu)\ .
 \end{equation*} 
We stress the fact that this choice can be considered as not standard, but it is rather natural in this context. We now reproduce, adapted to our setting, the result from \cite{Pa17} that will be the starting point of our investigation. To simplify the computations, we introduce for $\xi=(\xi_1,\xi_2)$
\begin{align*}
a(\xi):=(-1+e^{-2\pi i\xi_1}),\quad
b(\xi):=(-1+e^{-2\pi i\xi_2})\ .
\end{align*}

 \begin{Proposition}[{\cite[Prop. 3.5]{Pa17}}]\label{3}
 The operator $H_0$ is, by conjugation by $\U$, unitarily equivalent to a matrix-valued multiplication operator in $L^2(\T^2,\C^3)$ given by the real-analytic function
\begin{equation*}
h_0(\xi)=\begin{pmatrix}
m &a(\xi)& b(\xi)\\
\overline{a(\xi)}&-m &0\\
\overline{b(\xi)}&0&-m
\end{pmatrix}\ .
\end{equation*}
 \end{Proposition}
 
\subsection{Spectral theory for \texorpdfstring{$H_0$}{H0}}\label{subsec:Spectral}

To compute the   spectrum of $\U H_0\U^*$ we will obtain its band functions.  Because for every $\xi \in\T^2$, $h_0(\xi)$ has three eigenvalues, we will have three band functions. Note that, even if $h_0(\xi)$ is real-analytic, in contrast with the case when the base is $1$-dimensional, one can not a priori choose the band functions $\{\lambda_j\}$ to be analytic. We will have however an explicit expression for them and be able to compute $\sigma(H_0)=\bigcup_j \lambda_j(\T^2)$. 

One can see that for $\xi\in\T^2$, the characteristic polynomial associated to $h_0(\xi)$ is given by
\begin{equation}\label{eq:characteristic}
p_\xi(z)=(m-z)(m+z)^2+(m+z)\left(|b(\xi)|^2+|a(\xi)|^2\right)\ .
\end{equation}
We notice  the identities
\begin{equation*}
|a(\xi)|^2=2(1-\cos(2\pi\xi_1))=4\sin^2(\pi\xi_1) \text{ and } |b(\xi)|^2=2(1-\cos(2\pi\xi_2))=4\sin^2(\pi\xi_2)\ .
\end{equation*}
There are  three band functions:
\begin{equation*}
z_0(\xi)=-m\ , \quad z_\pm(\xi)=\pm\sqrt{m^2+|a(\xi)|^2+|b(\xi)|^2}\ .
\end{equation*}
This give us \cref{23sep21}.
For convenience we set $r_m(\xi):=m^2+|a(\xi)|^2+|b(\xi)|^2\geq0$ and assume the convention $r_0\equiv r$.

To compute the thresholds in $\mathcal{T}$   notice that
 \begin{equation}\label{eq:gradient1}
\nabla z_\pm(\xi)=\pm\frac{2\pi}{r_m(\xi)^\frac12}(\sin(2\pi\xi_1),\sin(2\pi\xi_2)).
 \end{equation}
 Then, both $z_+$ and $z_-$ have the same critical points and they give \cref{eq:tau1}.

First assume $m>0$. From  \cref{eq:gradient1} we see that  the thresholds  $\{\pm \sqrt{m^2+8},m\}$ are of \emph{elliptic} type. These thresholds are situated at the edges of the bands. The thresholds $\{\pm\sqrt{m^2+4}\}$ are of \emph{hyperbolic} type and  correspond to embedded thresholds (see \cite{IJ19} for this distinction for the Laplacian). Finally, there is $-m$ that corresponds to the degenerated eigenvalue, but it coincides also with the maximum of $z_-$. 

When $m=0$ there is no gap in the spectrum and  the thresholds are $\{0,\pm2,\pm\sqrt{8}\}$. However, since  $r((0,0))=0$, 
 $z_\pm$ are not everywhere differentiable, which indicates that $0$ is a  different  type of threshold. We call this  a \emph{Dirac point}.

 One of our focus of attention in the sequel will be the spectral properties near the energy corresponding to a degenerated eigenvalue of $H_0$. From the Floquet-Bloch theory described above, this eigenvalue appears due to the existence of a flat band. However, one can also understand directly in $l^2(X)$ its appearance: let suppose for simplicity that  $m= 0$. Let $x'\in \mathcal{V}$ be fixed and consider the closed path $\Gamma_{x'}\subset \mathcal{A}$ given by $\{(x',x'+\delta_1),(x'+\delta_1,x'+\delta_1+\delta_2),(x'+\delta_1+\delta_2,x'+\delta_2),(x'+\delta_2,x')\}$. We define $f_{x'}\in l^2(X)$ by $f_{x'}(\e)=1$ if $\e\in\Gamma_{x'}$ and such that it vanishes elsewhere. Remember that by definition of $C^1(X)$, we have that $f_{x'}(\overline{\e})=-1$ for $\e\in\Gamma_{x'}$. Since $f_{x'}$ vanishes in vertices we have $(H_0f_{x'})(\e)=0$ for every $\e$. Furthermore,
\begin{align*}
(H_0f_{x'})(x')=&-f_{x'}(x',x'+\delta_1)-f_{x'}(x',x'+\delta_2)-f_{x'}(x',x'-\delta_1)-f_{x'}(x',x'-\delta_2)\\
=&-1+1+0+0=0.
\end{align*}
Similar computations hold for $x'+\delta_1$, $x'+\delta_1+\delta_2$ and $x'+\delta_2$ while for the other vertices $f_{x'}$ vanishes in every edge around them. If follows that $H_0f_{x'}\equiv 0$. Repeating the procedure for different $x'\in \V$ we get an infinitely dimensional kernel.

\section{SSF representation}\label{sec:SSFrep}

Let us start this section by giving some notation that we will use henceforth. 

Denote by $\mathfrak{S}_\infty$ the class of compact operators.  For $K\in\mathfrak{S}_\infty$  self-adjoint  and $r>0$, we set
$$
n_{\pm}(r; K) : = \rank\mathds{1}_{(r,\infty)}(\pm K)\ .
$$
Thus,  the functions $n_{\pm}(\cdot; K)$ are respectively the
counting functions of the positive and negative eigenvalues of the
operator $K$. If $K$ is compact but not necessarily
self-adjoint we will use the notation
$$
n_*(r; K) : = n_+(r^2; K^* K), \quad r> 0;
$$
thus  $n_{*}(\cdot; K)$ is the counting function of the singular
values of $K$, which we denote  by $\{s_j(K)\}$. Since $ n_+(r^2; K^*K) = n_+(r^2; K K^*)$ 
\begin{equation}\label{oct6_2}
n_*(r; K) = n_*(r; K^*),  \quad r>0.
\end{equation}
Besides,  for $r_1, r_2>0$,  we have  the  Weyl inequalities
    \begin{equation}\label{weyl1}
    n_\pm(r_1 + r_2; K_1 + K_2) \leq n_\pm(r_1; K_1) + n_\pm(r_2; K_2),
    \end{equation}
where  $K_j$, $j=1,2$, are self-adjoint compact operators (see e.g. \cite[Theorem 9.2.9]{BS87}), as well as the Ky Fan inequality
    \begin{equation}\label{kyfan1}
    n_*(r_1 + r_2; K_1 + K_2) \leq n_*(r_1; K_1) + n_*(r_2;
K_2),
    \end{equation}
    for compact but not necessarily self-adjoint $K_j$, $j=1,2$, (see e.g. \cite[Subsection 11.1.3]{BS87}).

Denote by $\mathfrak{S}_p$, $p\in [1,\infty)$, the Schatten-von Neumann class of compact operators, equipped with the norm 
\begin{equation*}
\norm{K}_p^p=\sum s_j(K)^p\ .
\end{equation*} 
Further, since for $K\in\mathfrak{S}_p,$ $ \| K \|_p  = \left( -\int_0^{\infty} \, \d n_*(r; K) \, r^p \right)^{1/p}
    $,
  the  Chebyshev-type estimate 
    \begin{equation}\label{cheby}
    n_*(r; K) \leq r^{-p} \|K\|_p^p
    \end{equation}
    holds true for any $r > 0$ and $p \in [1, \infty)$.

Now, define for $z \in \C^+$ the operator $ K(z)= V^{1/2} (H_0-z)^{-1} V^{1/2}.$  By \cref{lemmalap} below, the norm limits 
\begin{equation}\label{def_K}\lim_{\delta\downarrow 0}K(\lambda+i\delta) =: K(\lambda+i0)\ee exist for any $\lambda\in \R\setminus\mathcal{T}$.   Moreover, recalling that $H_\pm=H_0\pm V$, by \cite[Theorem 1.1]{Pu98}  we have that  the SSF  
admits the representative \bk
\begin{equation}\label{2dec19}
    \xib(\lambda;H_\pm,H_0) = \pm \frac{1}{\pi}\int_\R\frac{\d t}{1+t^2} n_{\mp}(1;{\rm Re} \,K(\lambda+i 0)+t\, {\rm Im} \, K(\lambda +i0)).
\ee
\emph{All the results of this article about the SSF will concern its representative given by \cref{2dec19}.}

From this representation it is possible to obtain the following lemma, which is a small modification of {\cite[Lemma 2.1]{Pu98}}  that we state for the ease of the reader.

\begin{Lemma}\label{lepushbis}Let ${\rm T }(\lambda)$ be a family of finite rank operators. Then, for any $\epsilon \in (0,1)$

$$\xib(\lambda;H_\pm,H_0)\leq \pm n_\mp(1\mp\epsilon;{\rm Re}\,K(\lambda+i 0))+\frac{1
 }{\pi \varepsilon}||{\rm Im}\,K(\lambda+i 0)-{\rm T}(\lambda)||_{1}+\rank{\rm T}(\lambda)\ ;
$$
$$\xib(\lambda;H_\pm,H_0)\geq \pm n_\mp(1\pm\epsilon;{\rm Re}\,K(\lambda+i 0))-\frac{1
}{\pi \varepsilon}||{\rm Im}\,K(\lambda+i 0)-{\rm T}(\lambda)||_{1}-\rank{\rm T}(\lambda)\ .
$$
\end{Lemma}

Motivated by \cref{lepushbis}, we now look for a more explicit expression of the operator  $\U K(\lambda+i0) \U^*$. Throughout this section we will consider a fixed $\lambda\in \sigma(H_0)\backslash \Tau $, write $z=\lambda+i\delta$ and study the limit as $\delta\downarrow 0$. 
Furthermore, for simplicity we will assume $\lambda>0$ but, by symmetry, the results also holds for the lower band.

We have  that for $z\in\C\backslash \sigma(H_0)$
\begin{equation*}
(h_0-z)^{-1}(\xi)=\frac{1}{p_\xi(z)}\begin{pmatrix}
(m+z)^2 & (m+z)a(\xi)& (m+z)b(\xi)\\
(m+z)\overline{a(\xi)}&z^2-m^2-|b(\xi)|^2 &\overline{a(\xi)}b(\xi)\\
(m+z)\overline{b(\xi)}&a(\xi)\overline{b(\xi)}&z^2-m^2-|a(\xi)|^2
\end{pmatrix}, \
\end{equation*}
 the characteristic polynomial ${p_\xi(z)}$ given in 
\cref{eq:characteristic}.

We are interested in studying, for $z$ with $\Im(z)>0$,
$ 
\U V^{1/2} (H_0-z)^{-1}V^{1/2} \U^*.$ We can readily see that
\begin{equation*}
\begin{split}
\U V^{1/2}&(H_0-z)^{-1}V^{1/2}\U^* =\frac{\U V\U^*}{m+z}\begin{pmatrix}
0 &0 &0 \\[.5em]
0&-1&0\\[.5em]
0&0&-1\end{pmatrix}
\\
  +\U V^{1/2}\U^*&\frac{1}{(m+z)(r_m(\xi)-z^2)}\begin{pmatrix}
(m+z)^2&(m+z)a(\cdot) &(m+z)b(\cdot) \\[.5em]
(m+z)\overline{a(\cdot)}& |a(\cdot)|^2&\overline{a(\cdot)}b(\cdot)\\[.5em]
(m+z)\overline{b(\cdot)}&a(\cdot)\overline{b(\cdot)}&|b(\cdot)|^2\end{pmatrix}
 \U V^{1/2} \U^*\ . 
\end{split}
\end{equation*}
Let $g_j$ be defined by \cref{eq:littlev} but corresponding to the potential $G:=V^{1/2}$. Now, for each $\xi'\in\T^2$ and $z\in \C$ define $t_{z,m}(\xi'):L^2(\T^2;\C^3)\to\C^3$ by 
\begin{equation}\label{eq:deft}
\begin{pmatrix}f_1\\f_2\\f_3\end{pmatrix}\mapsto
\begin{pmatrix}
(m+z)\displaystyle{\int_{\T^2}\d\eta\ \hat{g}_1(\xi'-\eta)f_1(\eta)}\\[.7em]
a(\xi')\displaystyle{\int_{\T^2}\d\eta\ \hat{g}_2(\xi'-\eta)f_2(\eta)}
\\[.7em]
b(\xi')\displaystyle{\int_{\T^2}\d\eta\ \hat{g}_3(\xi'-\eta)f_3(\eta)}
\end{pmatrix}\ ,
\end{equation}
and
$t^{\dagger}_{z,m}(\xi'):\C^3 \to L^2(\T^2;\C^3)$ by 
\begin{equation}\label{eq:deftdaga}
\begin{pmatrix}\omega_1\\\omega_2\\\omega_3\end{pmatrix}\mapsto\begin{pmatrix}
\omega_1(m+z)\hat{g}_1(\cdot-\xi') \\\omega_2\overline{a(\xi')}\hat{g}_2(\cdot-\xi')\\\omega_3 \overline{b(\xi')}\hat{g}_3(\cdot-\xi')
\end{pmatrix}\ .
\end{equation}
We also note $t_{z,0}(\xi')=:t_{z}(\xi')$ and notice that $t^{\dagger}_{\overline{z},m}(\xi')=t_{z,m}(\xi')^*$. Then, on can see that
\begin{equation}\label{eq:antescambio}
\U V^{1/2}(H_0-z)^{-1}V^{1/2}\U^*=\frac{\U V\U^*}{m+z}\begin{pmatrix}
0 &0 &0 \\[.5em]%
0&-1&0\\[.5em]
0&0&-1\end{pmatrix}+\frac{1}{m+z}\int_{\T^2}\d\xi'\ \frac{t_{z,m}^\dagger(\xi') A t_{z,m}(\xi')}{(r_m(\xi')-z^2)} \ ,
\end{equation}
where 
\begin{equation*}
A=\begin{pmatrix}
1 &1 &1 \\[.5em]
1&1&1\\[.5em]
1&1&1\end{pmatrix}\ .
\end{equation*}
In \cref{eq:antescambio}  define 
 $T_{z,m}(\xi)=t_{z,m}^\dagger(\xi)A t_{z,m}(\xi).$ 
 As before, we also denote $T_{z,0}(\xi)=:T_{z}(\xi)$.

{ From now on,  we will assume that each $v_j$ satisfies \cref{3dec19_3} with $0<\beta_j\leq1$.} Further, we set \bel{1sep22}\widetilde{\beta}:=\min\{\beta_2,\beta_3\};\qquad\beta:=\min\{\beta_1,\widetilde{\beta}\}\ .\ee 
\begin{Lemma}\label{le_T}
For every $z\in \C$ and $\xi\in\T^2$ we have that $T_{z,m}(\xi)\in \mathfrak{S}_1(L^2(\T^2;\C^3))$. Furthermore, $T_{z,m}$  is locally H\"older with exponent $\beta$ in $\xi$ in the $\mathfrak{S}_1$-norm, and  it depends analytically in $z$. 
\end{Lemma}
\begin{proof}
First,  it is easy to see that
\begin{equation}\label{27jul20b}
\begin{split}
||t_{z,m}(\xi)||_2^2=&\int_{\T^2}\d\eta|m+z|^2|\hat{g}_1(\xi-\eta)|^2+|a(\xi)|^2 |\hat{g}_2(\xi-\eta)|^2+|b(\xi)|^2 |\hat{g}_3(\xi-\eta)|^2 \\
=&|m+z|^2\norm{\hat{g}_1}^2+|a(\xi)|^2 \norm{\hat{g}_2}^2+|b(\xi)|^2 \norm{\hat{g}_3}^2 \ ,
\end{split}
\ee
which in particular gives us that 
\begin{equation}\label{27jul20}\
||T_{z,m}(\xi)||_1= ||t_{z,m}^\dagger(\xi)  At_{z,m}(\xi)||_1
\leq C||t_{z,m}(\xi)||_2^2 < \infty \ .\ee
Now, let us fix $z\in \C$ and $m\geq0$. For $\xi$ and $\xi'$ in $\T^2$ we have
  \begin{equation}\label{18nov}||T_{z,m}(\xi')-T_{z,m}(\xi)||_1\leq  ||t_{z,m}^\dagger(\xi')-t_{z,m}^\dagger(\xi)||_2 ||A t_{z,m}(\xi') ||_2+||t_{z,m}^\dagger(\xi) A ||_2 ||t_{z,m}(\xi')-t_{z,m}(\xi)||_2. \ee
Using 
\begin{equation}\label{eq:lip}
 \begin{split}
 ||t_{z,m} (\xi')-t_{z,m}(\xi)||_2^2=&\int_{\T^2}\d\eta|m+z|^2|\hat{g}_1(\xi'-\eta)-\hat{g}_1(\xi-\eta)|^2\\
 &+ |a(\xi')\hat{g}_2(\xi'-\eta)-a(\xi)\hat{g}_2(\xi-\eta)|^2\\
 &+|b(\xi') \hat{g}_3(\xi'-\eta)-b(\xi) \hat{g}_3(\xi-\eta)|^2\ ,
\end{split}
\end{equation}
 together with the fact that condition \eqref{3dec19_3} ensures us  (see \cite[Theorem 3.5]{DDR19})
 \begin{equation}\label{16oct20}
 \int_{\T^2}\d\eta|\hat{g}_j(\eta+h)- \hat{g}_j(\eta)|^2=O(|h|^{2\beta_j}),\quad |h|\to 0, \quad 1\leq j \leq 3\ ,
 \end{equation}
we can conclude that, for $|\xi'-\xi|\to 0$,
\begin{align}
||t_{z,m} (\xi')-t_{z,m}(\xi)||_2^2\leq& C\(|m+z|^2\cdot|\xi'-\xi|^{2\beta_1}+|a(\xi')|^2|\xi'-\xi|^{2\beta_2}\right.\label{eq:18nov2previo}\\
&\left.+|b(\xi')|^2|\xi'-\xi|^{2\beta_3}+|a(\xi')-a(\xi)|^2\norm{\hat{g}_2}^2+|b(\xi')-b(\xi)|^2 \norm{\hat{g}_3}^2 \)\nonumber\\
\leq &C |\xi'-\xi|^{2\beta} \ .\label{18nov2}
\end{align}
Taking into account \cref{18nov} we obtain the Hölder property. Finally, the analyticity in $z$ can be directly observed in \cref{eq:deftdaga,eq:deft}.
\end{proof}

For $z\in \C\backslash \sigma(H_0)$ we want to use the coarea formula 
\bel{9dec21}\int_{\T^2}\d\xi\ \frac{T_{z,m}(\xi)}{(r_m(\xi)-z^2)} = \int_{m^2}^{M_m^2} \frac{ \d\rho}{\rho-z^2} \int_{r_m^{-1}(\rho)}\d\gamma\frac{T_{z,m}(\xi)}{|\nabla r_m (\xi)|}\ ,\ee
where $\d\gamma=\d\gamma_\rho$ is the one dimensional Hausdorff  measure over the level curve $r_m^{-1}(\rho)$, and $M_m:=\sqrt{m^2+8}$.  To this end it is enough to show that
\begin{equation*}
\frac{\lp T_{z,m }(\cdot)\rp_1}{(r_m(\cdot)-z^2)|\nabla r_m(\cdot)|} \text{ is in } L^1(\T^2)\ ,
\end{equation*}
which easily follows from the fact that $z\in \C\backslash \sigma(H_0)$, the boundedness  of $\lp T_{z,m }(\cdot)\rp_1$ and 
\begin{equation}\label{eq:nablarmodulo}
|\nabla r_m(\xi_1,\xi_2)|^2=16\pi^2(\sin^2(2\pi\xi_1)+\sin^2(2\pi\xi_2))\ .
\end{equation}

In \cref{fig:curva} we show the level curves of $r_m$. One can readily see a periodicity that will simplify some of our computations below. 

\begin{figure}
\centering
\begin{subfigure}{.6\textwidth}
  \centering
  \includegraphics[width=.95\textwidth]{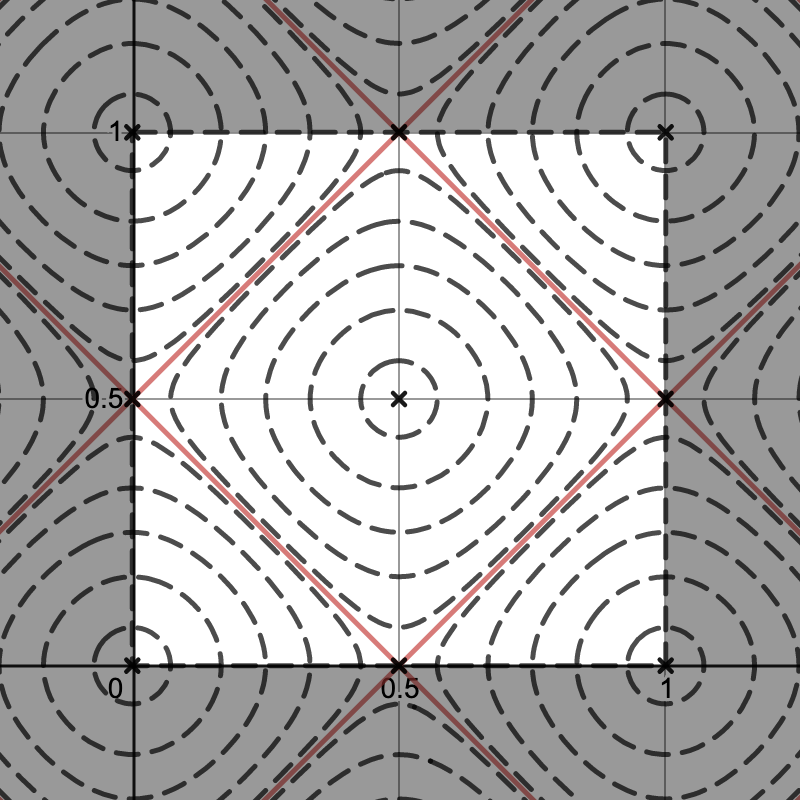}
\end{subfigure}
\caption{This figure shows the level curves for $r_m$. 
The level curve corresponding to the hyperbolic threshold is shown in red. Further, critical points are marked with a cross.}
\label{fig:curva}
\end{figure}

The main result of this section is the following proposition which together with \cref{eq:antescambio} gives an explicit description of $K(\lambda +i0)$. 

\begin{Proposition}\label{lemmalap}
Let $\lambda\in \R\backslash \mathcal{T} $ and set $z=\lambda+i\delta$. Then 
\begin{equation}\label{3dec19_1}
\begin{split}
\int_{m^2}^{M^2_m} \frac{\d\rho}{\rho-z^2} \int_{r_m^{-1}(\rho)}\d\gamma \frac{T_{z,m}(\xi)}{|\nabla r_m (\xi)|}\longrightarrow \text{p.v.} \int_{m}^{M_m} \d u\frac{2 u}{(u^2-\lambda^2)}&\int_{r_{m}^{-1}(u^2)} \d \gamma  \frac{T_{\lambda,m}(\xi)}{|\nabla r_m({\xi})|}\\
 -i\pi&\int_{r_m^{-1}(\lambda^2)} \d \gamma\frac{T_{\lambda,m}(\xi)}{|\nabla r_m({\xi})|}\ ,\end{split}
\end{equation}
as $\delta\downarrow 0$ in the $\mathfrak{S}_1$-norm.
\end{Proposition}

We present some preparatory lemmata before given the proof. After making the change $u^2=\rho$ in \cref{9dec21}, let us set
\begin{equation}\label{eq:defB}
B_m(u,z):=\frac{2u}{u+z}\int_{r_m^{-1}(u^2)}\d \gamma\frac{T_{z,m}(\xi)}{|\nabla r_m({\xi})|}\ .
\end{equation}

Note that this operator-valued function admits a limit when $\delta\downarrow 0$ and so we can write $B(u,\lambda)$. Moreover, $B(\cdot,\lambda)$ satisfies the following lemma:

\begin{Lemma}\label{le_gradient} Outside $\mathcal{T}$, the function $u \mapsto B_m(u,\lambda)$ is locally Hölder with exponent $\beta$ (defined in \cref{1sep22}).
\end{Lemma}
\begin{proof}

Let us consider an interval $I \Subset(m,\sqrt{m^2+4})$. For $u\in I$, and taking the function $\arcsin:[-1,1]\to[-\frac{\pi}2,\frac{\pi}2]$ we define 
 \begin{equation*}
 f_u(\xi_1)=\frac1\pi\arcsin\(\sqrt{\frac{u^2-m^2}{4}-\sin^2(\pi\xi_1)}\)\ ,
 \end{equation*}
 for $|\xi_1|\leq \frac1\pi\arcsin\(\sqrt{\frac{u^2-m^2}4}\)=:\xi_1(u)$.
 Since $(\xi,\pm f_u(\xi_1))$ parametrize $r_m^{-1}(u^2)$ we obtain:

 \begin{align*}
 \int_{r_m^{-1}(u^2)}\d \gamma \frac{T_{\lambda,m}(\xi)}{|\nabla r_m (\xi)|} 
 = \int_{-\xi_1(u)}^{\xi_1(u)}\d \xi_1 &\( \frac{|\nabla r (\xi_1,f_u(\xi_1))|}{|\partial_2 r(\xi_1,f_u(\xi_1))|}\times\frac{T_{\lambda,m}(\xi_1,f_u(\xi_1)}{|\nabla r (\xi_1,f_u(\xi_1))|}\right.\\
 &\left.- \frac{|\nabla r (\xi_1,-f_u(\xi_1))|}{|\partial_2 r(\xi_1,-f_u(\xi_1))|}\times\frac{T_{\lambda,m}(\xi_1,-f_u(\xi_1)}{|\nabla r (\xi_1,-f_u(\xi_1))|}\)\\
 =\frac{1}{8\pi}\int_{-\xi_1(u)}^{\xi_1(u)}\d \xi_1&\(\frac{T_{\lambda,m}(\xi_1,f_u(\xi_1)}{\sqrt{F_u(\xi_1)(1-F_u(\xi_1))}}+\frac{T_{\lambda,m}(\xi_1,-f_u(\xi_1)}{\sqrt{F_u(\xi_1)(1-F_u(\xi_1))}}\)\ ,
 \end{align*}

 where 
 \begin{equation*}
 F_u(\xi_1)=\frac{u^2-m^2}{4}-\sin^2(\pi\xi_1)\ .
 \end{equation*}

Now, let us consider $v<u$ in $I$ and bound $\|B(u,\lambda)-B(v,\lambda)\|$:
\begin{align*}
 &\norm{\frac{2u}{u+\lambda}\int_{-\xi_1(u)}^{\xi_1(u)}\d \xi_1 \frac{T_{\lambda,m}(\xi_1,f_u(\xi_1))}{\sqrt{F_u(\xi_1)(1-F_u(\xi_1))}}-\frac{2v}{v+\lambda}\int_{-\xi_1(v)}^{\xi_1(v)}\d \xi_1 \frac{T_{\lambda,m}(\xi_1,f_v(\xi_1))}{\sqrt{F_v(\xi_1)(1-F_v(\xi_1))}}}_1 \\
&\leq \frac{2u}{u+\lambda}\int_{-\xi_1(u)}^{\xi_1(u)}\d \xi_1 \frac{\norm{T_{\lambda,m}(\xi_1,f_u(\xi_1))-T_{\lambda,m}(\xi_1,f_v(\xi_1))}_1}{\sqrt{F_u(\xi_1)(1-F_u(\xi_1))}}\\
&+ \int_{-\xi_1(u)}^{\xi_1(u)}\d \xi_1 \left| \frac{2u}{(u+\lambda)\sqrt{F_u(\xi_1)(1-F_u(\xi_1))}}-\frac{2v}{(\lambda+v)\sqrt{F_v(\xi_1)(1-F_v(\xi_1))}}\right|\hspace*{-1pt}\lVert T_{\lambda,m}(\xi_1,f_v(\xi_1)) \rVert_1\\
&+ \int_{\[-\xi_1(u),\xi_1(u)\]\backslash \[-\xi_1(v),\xi_1(v)\]}\d \xi_1 \norm{\frac{2vT_{\lambda,m}(\xi_1,f_v(\xi_1))}{(v+\lambda)\sqrt{F_v(\xi_1)(1-F_v(\xi_1))}} }_1\ .
\end{align*}
First we notice that 
\begin{equation*}
u\to \int_{-\xi_1(u)}^{\xi_1(u)}\frac{\d \xi_1}{\sqrt{F_u(\xi_1)(1-F_u(\xi_1))}}
\end{equation*} 
is bounded on $I$ since
\begin{equation}\label{12jan21}
0<\min_{u\in I}\{\min_{|\xi_1|\leq\xi_1(u)}\{ F_u(\xi_1)\}\}\text{ and }\max_{u\in I}\{\max_{|\xi_1|\leq\xi_1(u)}\{ F_u(\xi_1)\}\}<1\ . 
\end{equation}

Then, to treat the first term, if  we  take $I$ suitably small,  one can use \cref{le_T} to obtain
\begin{equation*}
\norm{T_{\lambda,m}(\xi_1,f_u(\xi_1))-T_{\lambda,m}(\xi_1,f_v(\xi_1))}_1\leq C(f_u(\xi_1)-f_v(\xi_1))^\beta\ .
\end{equation*} 
One concludes by noticing that $f_u$ is Lipschitz in $u$, for $u\in I$, with a uniform constant in $\xi_1$. For the second term, from \cref{12jan21},
we have that 
$\frac{2u}{(u+\lambda)\sqrt{F_u(\xi_1)(1-F_u(\xi_1))}}\in C^1(I)$. Using the same argument one can check that the norm in the third term is uniformly bounded and then conclude by noticing that $u\to\xi_1(u)$ has a bounded derivative on $I$.
\end{proof}\bk
As a direct consequence, by applying Sokhotski-Plemelj formula we obtain the following.
\begin{Corollary}\label{lemmalap2}
Let $\lambda$ and $z=\lambda+i\delta$ be as before. Then 
\begin{equation*}
\lp\int_m^{M_m}\d u \frac{B_m(u,\lambda)}{u-z}-\text{p.v.} \int_{m}^{M_m}\d u\frac{B_m(u,\lambda)}{u-\lambda} -i\pi B_m(\lambda,\lambda)\rp_1=O(\delta) \text{ as } \delta\downarrow 0 \ .
\end{equation*}
\end{Corollary}

The last lemma we need is the following.
\begin{Lemma}\label{lemmalap1}
Let $\lambda$ and $z=\lambda+i\delta$ be as before. Then 
\begin{equation*}
\int_m^{M_m}\d u\lp
\frac{B_m(u,z)-B_m(u,\lambda)}{u-z}
\rp_1
\rightarrow 0 \text{ as } 
\delta\downarrow 0 \ .
\end{equation*}
\end{Lemma}

\begin{proof}
We start by noticing that
\begin{align}
&\lp 
B_m(u,z)  -B_m(u,\lambda)
\rp_1
 = \lp \frac{2u}{u+z}\int_{r_m^{-1}(u^2)}\d \gamma \frac{T_{z,m}(\xi)}{|\nabla r_m(\xi)|}
-\frac{2u}{u+\lambda}\int_{r_m^{-1}(u^2)}\d \gamma \frac{T_{\lambda,m}(\xi)}{|\nabla r_m(\xi)|}  
\rp_1 \nonumber
\\
&\leq 2u
\(
\lp \(\frac{1}{u+z}-\frac{1}{u+\lambda}\) \int_{r_m^{-1}(u^2)}\hspace*{-8pt}\d \gamma \frac{T_{z,m}(\xi)}{|\nabla r_m(\xi)|}  \rp_1
 +\lp \frac{1}{u+\lambda} \int_{r_m^{-1}(u^2)}\hspace*{-8pt}\d \gamma \frac{T_{z,m}(\xi)-T_{\lambda,m}(\xi)}{|\nabla r_m(\xi)|} \rp_1
 \)\nonumber\\
&\leq \frac{2u\delta}{|u+z||u+\lambda|} \lp  \int_{r_m^{-1}(u^2)}\d \gamma \frac{T_{z,m}(\xi)}{|\nabla r_m(\xi)|} \rp _1
+ \frac{2u}{|u+\lambda|} \lp \int_{r_m^{-1}(u^2)}\d \gamma \frac{T_{z,m}(\xi)-T_{\lambda,m}(\xi)}{|\nabla r_m(\xi)|}  \rp_1. \label{eq:dosterminos}
\end{align} 
We now treat each term of \cref{eq:dosterminos} separately. For the first term one only needs \cref{le_gradient} and  check that \begin{equation*}
\(\sup_{u\in \[m,M_m\]}\frac{2u}{|u+z||u+\lambda|}\)\(\sup_{\xi\in\T^2} \lp T_{z,m}(\xi)\rp_1\) < \infty\ ,
\end{equation*}
uniformly on $\delta$.  For the second term, \cref{le_T} gives us that
\begin{equation*}
\lp B_m(u,z)  - B_m(u,\lambda) \rp_1 \leq C \delta\int_{r_m^{-1}(u^2)}\frac{\d\gamma}{|\nabla r_m(\xi)|}=: C \delta R(u) \ .
\end{equation*}
Note that $R:\[m;M_m\]\setminus \Tau\to \R$ remains bounded near $\{m,M_m\}$ (see \cref{nabla} below). Furthermore, for $\epsilon>0$, it satisfies $R(\sqrt{m^2+4}\pm\epsilon)=O(\ln(\epsilon^{-1}))$. Hence $R\in L^1(\[m;M_m\])$. The result follows by applying the dominated convergence theorem once one notice that 
\begin{equation*}
\frac{\delta R(u)}{\sqrt{(u-\lambda)^2+\delta^2}} \to 0\text{ as }\delta\to 0,
\end{equation*}
for all $u\notin\{\lambda,\sqrt{m^2+4}\}$. 
\end{proof}

We have gathered all the result needed for proving \cref{lemmalap}.

\begin{proof}[Proof of \cref{lemmalap}]
Using a change of variable one can see that the result follows if 
\begin{equation*}
\lp \int_{m}^{M_m}\d u\frac{B_m(u,z)}{u-z}-\text{p.v.} \int_{m}^{M_m}\d u\frac{B_m(u,\lambda)}{u-\lambda} -i\pi B_m(\lambda,\lambda) \rp\to 0 \text{ as } \delta\downarrow 0 .
\end{equation*}
This follows from the fact that it is bounded by 
\begin{equation*}
\lp \int_{m}^{M_m}\d u\frac{B_m(u,z)-B_m(u,\lambda)}{u-z}\rp+\lp\int_{m}^{M_m}\d u\frac{B_m(u,\lambda)}{u-z}-\text{p.v.} \int_{m}^{M_m}\d u\frac{B_m(u,\lambda)}{u-\lambda} -i\pi B(\lambda,\lambda) \rp
\end{equation*}
once one takes into account \cref{lemmalap1,lemmalap2}.
\end{proof}

As a consequence of \cref{lemmalap}, for  $\lambda\in \R\backslash \mathcal{T} $ we have
\begin{equation}\label{21jan20}
\U[{\rm Im} \, K(\lambda +i0)]\U^*=\frac{-\pi}{m+\lambda} \int_{r_m^{-1}(\lambda^2 ) }\d\gamma\frac{T_{\lambda,m}(\xi)}{|\nabla r_m (\xi)|}\ ,\ee
and
\begin{equation}\label{21jan20_a}
\begin{split}\U[{\rm Re} \, K(\lambda +i0)]\U^*=&\frac{\U V\U^*}{m+\lambda}\begin{pmatrix}
0 &0 &0 \\[.5em]
0&-1&0\\[.5em]
0&0&-1\end{pmatrix}\\
&+\frac{1}{m+\lambda}\text{p.v.} \int_{m}^{M_m} \d u\frac{2 u}{(u^2-\lambda^2)}\int_{r_{m}^{-1}(u^2)} \d \gamma  \frac{T_{\lambda,m}(\xi)}{|\nabla r_m({\xi})|} \ .\end{split}
\end{equation}

\section{Preliminary computations: parabolic thresholds and Dirac point case }\label{sec:alphaenergy}

Having in mind \cref{lepushbis}, in this section we study some asymptotic properties of the  operators given by  \cref{21jan20,21jan20_a},   for $\lambda$ close to  $\pm m$. When $m\neq 0$ and $|\lambda|<m$, the imaginary part obtained in \cref{21jan20} vanishes and the p.v. part of \cref{21jan20_a} is just the  integral over the interval $[m,M_m]$ . Hence, we start by considering in detail the case that $\lambda \to \pm m$ with $|\lambda|>m$. Several notations will depend on a $\pm$ subscript. We do this to shorten notations and study both limits, $\pm m$, at the same time. Notice that the case $m=0$ needs also to be considered. In that case, the subscript $\pm$ will indicate if we are approaching $0$ by the positive or negative side. 

\begin{Lemma}\label{imaginaryalpha}
The following relation holds true 
$$\left|\left|\int_{r^{-1}_m(\lambda^2) }\d\gamma\frac{T_{\lambda,m}(\xi)}{|\nabla r_m (\xi)|}\right|\right|_1=O(\lambda(m+\lambda)), \quad \lambda\to \pm m, |\lambda|>m.$$
\end{Lemma}
\begin{proof} 
First notice that from \cref{27jul20b,27jul20}, for all $\lambda$ and $\xi$
\begin{equation}\label{18nov6alpha}
||T_{\lambda,m}(\xi)||_1\leq C\((m+\lambda)^2+ r (\xi)\) .
\end{equation}
Now, recalling that 
\begin{equation*}
r(\xi)=2(1-\cos(2\pi\xi_1))+2(1-\cos(2\pi\xi_2))=4\sin^2(\pi \xi_1)+4\sin^2(\pi \xi_2)
\end{equation*}
and comparing it with \cref{eq:nablarmodulo} one can easily deduce that  for  $\xi\in [-1/4,1/4]\times[-1/4,1/4]$
\begin{equation}\label{nabla}
\frac{1}{16\pi}|\nabla r(\xi)|^2\leq {r(\xi)}\leq |\nabla r(\xi)|^2\ .\end{equation}
Furthermore, using 
\begin{equation}\label{eq:reduccion}
\nabla r_m(\xi)=\nabla r(\xi)  \text{ and } \quad r_m^{-1}(u^2)=r^{-1}(u^2-m^2) \  ,
\end{equation} 
 we see that for $\lambda$ near $\pm m$, the variable $\xi\in r_m^{-1}(\lambda)= r^{-1}(\lambda-m^2)$ is small in modulus. Thus from \cref{18nov6alpha,nabla} 
\begin{align*}
\left|\left|\int_{r^{-1}_m(\lambda^2)}\d\gamma\frac{T_{\lambda,m}(\xi)}{|\nabla r (\xi)|}\right|\right|_1\leq C \int_{r^{-1}(\lambda^2-m^2) }\d \gamma\, \frac{(m+\lambda)^2+r(\xi)}{r(\xi)^{1/2}}\leq &C\( (m+\lambda)^2+(\lambda^2-m^2)\)\\
=&C\lambda(m+\lambda)\ . \qedhere
\end{align*}
\end{proof}
From \cref{21jan20_a}, we define 
\begin{equation}\label{eq:Slambdaalpha}
{\mathcal{S}}_0(\lambda):=\text{p.v.}\int_m^{M_m}\d u \frac{2u}{u^2-\lambda^2} \int_{r_{m}^{-1}(u^2)}\d\gamma\frac{T_{\lambda, m}(\xi)}{|\nabla r (\xi)|} .
\end{equation}
Now, let us set
\begin{equation*}
B^{\pm}_{\lambda,m}(u):=\frac{2u}{(u\pm\lambda)}\int_{r^{-1}_{m}(u^2)}\d\gamma \frac{T_{\lambda,m}(\xi)}{|\nabla r(\xi)|} \ ,
\end{equation*}

together with 
\begin{equation*}{\mathcal{S}}_{1}(\lambda):=\begin{cases}\displaystyle{\int_{\[m,M_m\]}\d u \frac{B^{\pm}_{\lambda,m}(u)}{u\mp\lambda}};&0\leq\pm\lambda<m\\
\displaystyle{\int_{\[m,\frac{m\pm\lambda}{2}\]\cup\[\frac{\pm 3\lambda-m}{2},M_m\]}\d u \frac{B^{\pm}_{\lambda,m}(u)}{u\mp\lambda}};&\pm\lambda>m
\end{cases}\ .\end{equation*}

 \begin{Lemma}\label{lemma:pvalpha0}
 \begin{equation*}
 ||{\mathcal{S}}_{0}(\lambda)-{\mathcal{S}}_{1}
 (\lambda)||_1=O(\lambda(m+\lambda)), \quad \lambda\to \pm m.
\end{equation*}
 \end{Lemma}
\begin{proof}

Let us first notice that
\begin{equation}\label{eq:s2tilde}
{\mathcal{S}}_{0}(\lambda)-{\mathcal{S}}_{1}(\lambda)=\begin{cases}0;&|\lambda|<m\\\displaystyle{\int_0^{
\frac{\pm\lambda-m}{2}}\frac{\d\rho}{\rho}\(B^\pm_{\lambda,m}(\rho\pm\lambda)-B^\pm_{\lambda,m}(-\rho\pm\lambda)\)};&\pm\lambda>m\end{cases} \ .
\end{equation}
Thus, we only need to consider the cases where $|\lambda|>m$. \bk
To prove this lemma we will write ${\mathcal{S}}_{0}-{\mathcal{S}}_{1}$ as the sum of five terms, $R_j$, $j=0,1,2,3, 4$, the first of which is 

$$R_0(\lambda):=\int_0^{\frac{\pm\lambda-m}{2}}\frac{\d \rho }{\rho}  \left(\frac{2\rho\pm2\lambda}{\rho\pm2\lambda}-\frac{-2\rho\pm2\lambda}{-\rho\pm2\lambda}\right)\int_{r^{-1}_m((\pm\lambda+\rho)^2)}\d\gamma \frac{T_{\lambda,m}(\xi)}{\nabla r (\xi)}\ .$$
 Then, using   \cref{18nov6alpha,nabla}  we can see that  \bk
\begin{align*}
||R_0(\lambda)||_1\leq & \int_0^{\frac{\pm\lambda-m}{2}}\frac{\d \rho }{\rho} \frac{\pm4\lambda\rho}{(\rho\pm2\lambda)(-\rho\pm2\lambda)} \int_{r^{-1}((\pm\lambda+\rho)^2-m^2)}\d\gamma \frac{||T_{\lambda,m} (\xi)||_1}{|\nabla r (\xi)|}\\
\leq& \frac{C}{\pm3\lambda +m}\int_0^{\frac{\pm\lambda-m}{2}}\d \rho \int_{r^{-1}((\pm\lambda+\rho)^2-m^2)}\d\gamma  \(\frac{(m+\lambda)^2}{r(\xi)^{\frac12}}+r(\xi)^{\frac12}  \)\\
\leq& \frac{C}{\pm3\lambda +m}\int_0^{\frac{\pm\lambda-m}{2}}\d \rho\( (m+\lambda)^2+(\pm\lambda+\rho)^2-m^2\)\  .
\end{align*}

From the last expression we can deduce that
\begin{equation}\label{eq:stilde0}
||R_0(\lambda)||_1=\begin{cases}
O((\lambda+m)^2) &\text{ if }m \neq 0 , \lambda\uparrow -m\\
O((\lambda-m)) &\text{ if }m \neq 0 , \lambda\downarrow m\\
O(\lambda^2) &\text{ if }m = 0 , \lambda\to 0
\end{cases}\ .
\end{equation}
In order to write the decomposition of ${\mathcal{S}}_{0}(\lambda)-{\mathcal{S}}_{1}(\lambda)-R_0(\lambda)$ 
set $$\theta:=\frac{(\pm\lambda+\rho)^2-(\pm\lambda-\rho)^2}{4}=\pm\lambda\rho>0\ ,$$  and 
\begin{equation*}
D_\theta(\xi_1,\xi_2)=\begin{cases}\frac1{2\pi}\Big(\arccos(-\theta+\cos(2\pi\xi_1)),\arccos(-\theta+\cos(2\pi\xi_2))\Big) &\text{ if } \xi_1\geq 0 \text{ and } \xi_2\geq 0\ ,\\
\frac1{2\pi}\Big(-\arccos(-\theta+\cos(2\pi\xi_1)),\arccos(-\theta+\cos(2\pi\xi_2))\Big) &\text{ if } \xi_1< 0 \text{ and } \xi_2\geq 0\ ,\\
\frac1{2\pi}\Big(-\arccos(-\theta+\cos(2\pi\xi_1)),-\arccos(-\theta+\cos(2\pi\xi_2))\Big) &\text{ if } \xi_1\leq 0 \text{ and } \xi_2< 0\ ,\\
\frac1{2\pi}\Big(\arccos(-\theta+\cos(2\pi\xi_1)),-\arccos(-\theta+\cos(2\pi\xi_2))\Big) &\text{ if } \xi_1\geq 0 \text{ and } \xi_2 <0\ .
\end{cases}
\end{equation*}
  One can readily check that for $\rho\neq \lambda$, $ D_\theta$ is injective on $r^{-1}_m((\lambda-\rho)^2)$ and satisfies  $D_\theta(r^{-1}_m((\lambda-\rho)^2))\subset r^{-1}_m((\lambda+\rho)^2).$ Let us denote by $D_\theta^*d\gamma$ the pullback measure induced by $D_\theta$ on $r_m^{-1}((\lambda-\rho)^2)$ and set $\gamma_{R}:=r^{-1}_m((\lambda+\rho)^2)\setminus D_\theta(r^{-1}_m((\lambda-\rho)^2))$. 
 
In consequence, we can write  $\mathcal{S}_0-\mathcal{S}_{1}-R_0=\sum_{i=1}^4R_j$ where 
 
 \begin{align*}
R_1(\lambda):=&\int_0^{\frac{\pm\lambda-m}{2}} \frac{\d \rho }{\rho}\left(\frac{-2\rho\pm2\lambda}{-\rho\pm2\lambda}\right)\left(
\int_{ r^{-1}_m((\pm\lambda-\rho)^2)}D_\theta^*\d\gamma\,\frac{(T_{\lambda,m}(D_\theta(\xi))-T_{\lambda,m}(\xi))}{|\nabla r (D_\theta\xi)|}\right)\\
 R_2(\lambda):=&\int_0^{\frac{\pm\lambda-m}{2}} \frac{\d \rho }{\rho}\left(\frac{-2\rho\pm2\lambda}{-\rho\pm2\lambda}\right)\left(
\int_{\gamma_R}\d\gamma\,\frac{T_{\lambda,m}(\xi))}{|\nabla r (\xi)|}\right) \\
R_3(\lambda):=&\int_0^{\frac{\pm\lambda-m}{2}} \frac{\d \rho }{\rho}\left(\frac{-2\rho\pm2\lambda}{-\rho\pm2\lambda}\right)\left(
\int_{ r^{-1}_m((\pm\lambda-\rho)^2)}(D_\theta^*\d\gamma-\d\gamma)\frac{(T_{\lambda,m}(\xi))}{|\nabla r (D_\theta\xi)|}\right)\nonumber\\
	R_4(\lambda):=&\int_0^{\frac{\pm\lambda-m}{2}} \frac{\d \rho }{\rho}\left(\frac{-2\rho\pm2\lambda}{-\rho\pm2\lambda}\right)\left(
	\int_{ r^{-1}_m((\pm\lambda-\rho)^2)}\d\gamma \, T_{\lambda,m}(\xi)\(\frac{1}{|\nabla r (D_\theta\xi)|}-\frac{1}{|\nabla r (\xi)|}\)\right)\ .
\end{align*}
To estimate $R_1$ compute
\begin{equation}\label{18nov2aalpha}
\begin{array}{ll}
|D_\theta(\xi)-\xi|^2&=\displaystyle{\left|
\int_{-\theta+\cos(2\pi\xi_1)}^1\frac{\d t}{\sqrt{1-t^2}}-\xi_1\right|^2+ \left|
\int_{-\theta+\cos(2\pi\xi_2)}^1\frac{\d t}{\sqrt{1-t^2}}-\xi_1\right|^2}\\[1.5em]
&\leq\displaystyle{\left|\int_{\cos(2\pi\xi_1)-\theta}^{\cos(2\pi\xi_1)}\frac{\d t}{(1-t)^{1/2}(1+t)^{1/2}}\right|^2+\left|\int_{\cos(2\pi\xi_2)-\theta}^{\cos(2\pi\xi_2)}\frac{\d t}{(1-t)^{1/2}(1+t)^{1/2}}\right|^2}\\[1.5em]
&\leq C \theta,\end{array}
\end{equation} 
since $\xi_1$, $\xi_2$ are small for $\lambda$ near $\pm m$. Particularizing \cref{18nov,27jul20b,eq:18nov2previo} we can also obtain
\begin{align*}
||T_{\lambda,m}(D_\theta(\xi))-T_{\lambda,m}(\xi)||_1&\leq C\(2|m+\lambda|^2+r(D_\theta(\xi))+r(\xi)\)^{\frac12}\\
&\(|m+\lambda|\,|D_\theta(\xi)-\xi|^{\beta_1}+r(D_\theta(\xi))^{\frac12}|D_\theta(\xi)-\xi|^{\widetilde{\beta}}+|D_\theta(\xi)-\xi| \)
\end{align*}
and hence
\begin{equation*}
||R_1(\lambda)||_1 \leq C \int_0^{\frac{\pm\lambda-m}{2}}\hspace{-2pt}\d \rho \frac{(4\lambda(\lambda+m)+2\rho^2)^\frac12}{\rho}\int_{ r^{-1}((\pm\lambda+\rho)^2-m^2)}\hspace{-5.57pt}\d\gamma\frac{|m+\lambda|\theta^{\frac{\beta_1}{2}}+r(\xi)^{\frac12}\theta^{\frac{{\widetilde{\beta}}}{2}}+\theta^{\frac12}}{|\nabla r (\xi)|}\ .\\
\end{equation*}
From the last expression we can estimate each term and deduce
\begin{equation}\label{eq:stilde1}
R_1(\lambda)=\begin{cases}
O(|\lambda+m|) &\text{ if }m \neq 0 , \lambda\uparrow -m\\
O(|\lambda-m|^{\frac{\beta_1}2}) &\text{ if }m \neq 0 , \lambda\downarrow m\\
O(\lambda^{2}) &\text{ if }m = 0 , \lambda\to 0
\end{cases}\ .
\end{equation}

To estimate $R_2$ we need first an estimate on $\gamma_R$. Since $D_\theta$ does not change the quadrant of $\xi\in r^{-1}_m((\lambda-\rho)^2)$, $\gamma_R$ is composed of 8 parts. More precisely, if  $(\xi_1^-,0)\in r^{-1}_m((\lambda-\rho)^2) $  and $(\xi_1^+,0)\in r^{-1}_m((\lambda+\rho)^2) $, one of these parts is the section of $r^{-1}_m((\lambda+\rho)^2) $ going from $D_\theta((\xi_1^-,0))$ to $(\xi_1^+,0)$. Hence, the convexity of the level curves gives us 
\begin{align*}
\int_{\gamma_R}\d \gamma\leq & C\( |\arccos(1-\theta)|+\left|\arccos\(1+\tfrac{m^2-(\pm\lambda+\rho)^2}2\)-\arccos\(1+\tfrac{m^2-\lambda^2-\rho^2}2\)\right|\)\\ 
\leq & C\theta^{\frac12}\ .
\end{align*}
Using this estimate we compute
\begin{align*}
\|R_2(\lambda)\|_1\leq& \int_0^{\frac{\pm\lambda-m}{2}} \frac{\d \rho }{\rho}\left(\frac{-2\rho\pm2\lambda}{-\rho\pm2\lambda}\right)\left(
\int_{\gamma_R}\d\gamma\,\frac{\|T_{\lambda,m}(\xi))\|_1}{|\nabla r (\xi)|}\right) \\
\leq& C|\lambda|^{\frac12} \int_0^{\frac{\pm\lambda-m}{2}} \frac{\d \rho }{\rho^{\frac12}}
\frac{2m\lambda+2\lambda^2+\rho^2\pm 2\lambda\rho}{(\lambda^2-m^2\pm2\lambda\rho+\rho^2)^{\frac12}} \\
\leq& C|\lambda|^{\frac12} \int_0^{\frac{\pm\lambda-m}{2}}\d \rho \(\frac{\rho^2\pm 2\lambda\rho +\lambda^2-m^2 }{\rho}\)^{\frac12}
\frac{2m\lambda+2\lambda^2+\rho^2\pm 2\lambda\rho}{(\lambda^2-m^2\pm2\lambda\rho+\rho^2)} \\
\leq& C|\lambda|^{\frac12}\(1+\frac{\lambda+m}{\lambda-m}\) \int_0^{\frac{\pm\lambda-m}{2}}\d \rho (\rho\pm 2\lambda)^{\frac12}+\(\frac{\lambda^2-m^2 }{\rho}\)^{\frac12}\ .\\
\end{align*}
From this one gets
\begin{equation}\label{eq:stildeR2}
||R_2(\lambda)||_1=\begin{cases}
O(|\lambda+m|) &\text{ if }m \neq 0 , \lambda\uparrow -m\\
O(1) &\text{ if }m \neq 0 , \lambda\downarrow m\\
O(\lambda^2) &\text{ if }m = 0 , \lambda\to 0
\end{cases}\ .
\end{equation}
Next, by \cref{27jul20,27jul20b} we get
\begin{align}
||R_3(\lambda)||_1&\leq C \int_0^{\frac{\pm\lambda-m}{2}} \frac{\d\rho}{\rho}\int_{r^{-1}((\pm\lambda-\rho)^2-m^2)}\hspace*{-16pt}|D_\theta^*\d\gamma-\d\gamma|\frac{(m+\lambda)^2+r(\xi)}{r(D_\theta(\xi))^{1/2}}\nonumber \\
&\leq C  \int_0^{\frac{\pm\lambda-m}{2}}\d\rho \frac{2\lambda^2+2m\lambda\mp2\lambda\rho+\rho^2}{\rho((\pm\lambda+\rho)^2-m^2)^{\frac12}}\int_{r^{-1}((\pm\lambda-\rho)^2-m^2)}|D_\theta^*\d\gamma-\d\gamma| \ .\label{eq:s2alpha}
\end{align}

Let    $(\xi_1(t),\xi_2(t))$ be a parametrization of the curve $r^{-1}((\lambda-\rho)^2-m^2)$. It  is easy to see that 
\begin{align*}
|D_\theta^*\d\gamma&-\d\gamma|=
\left|\((\partial_1D_{\theta,1})(\xi(t))^2\xi_1'(t)^2+(\partial_2D_{\theta,2})(\xi_2(t))^2\xi_2'(t)^2\)^{1/2}-(\xi_1'(t)^2+\xi_2'(t)^2)^{1/2}\right|\d t\nonumber\ ,\end{align*}
where  $D_\theta(\xi)=(D_{\theta,1}(\xi),D_{\theta,2}(\xi))$. Notice that  
$$(\partial_1D_{\theta,1})(\xi)=\frac{\sin (2\pi\xi_1)}{(1-(\cos(2\pi\xi_1)-\theta)^2)^{1/2}};\quad (\partial_2D_{\theta,2})(\xi)=\frac{\sin (2\pi\xi_2)}{(1-(\cos(2\pi\xi_2)-\theta)^2)^{1/2}}.$$
Then, setting
\begin{equation*}
Q_\theta(\xi):=1-\frac{\sin^2 (2\pi\xi_1)}{|1-(\cos(2\pi\xi_1)-\theta)^2|} +1-\frac{\sin^2 (2\pi\xi_2)}{|1-(\cos(2\pi\xi_2)-\theta)^2|}\ ,
\end{equation*}
it is easy to see that 
\bel{18nov5alpha}|D_\theta^*\d\gamma-\d\gamma|\leq  Q_\theta(\xi)\d\gamma\ .\ee

Since $\theta$ is positive and $\xi$ is in a small neighborhood of $(0,0)$, we can  see that there exist positive constants $C_1$ and $C_2$ such that 
\begin{equation}\label{8sep20balpha}
Q_\theta(\xi)\leq C \theta\left(\frac{1}{\sin^2(2\pi\xi_1)+C_1\theta}+\frac{1}{\sin^2(2\pi\xi_2)+C_2\theta}\right).
\end{equation} 
Now, write $u^2=(\pm\lambda-\rho)^2$ and parametrize the part of $ r_m^{-1}(u^2)$ in $[0,1/2]\times[0,1/2]$ by 
$$\xi_1(t)=\frac{1}{2\pi}\arccos\(1-\frac{t}{2}\);\quad \xi_2(t)=\frac{1}{2\pi}\arccos\(\frac{t-u^2}{2}+1\),$$
for $t\in(0,u^2)$. Notice that 
$$ 
\xi'_1(t)\asymp\frac{1}{\sqrt{t}},\quad \xi'_2(t)\asymp \frac{1}{\sqrt{u^2-t}};\,\,\text{for}\,\,t \in(0,u^2),$$
 therefore 
\begin{equation}\label{9sep20alpha}\begin{split}
&d\gamma\asymp \frac{\d t}{\sqrt{t}},\,\,\text{for}\,\,t\in(0,u^2/2),\\&  \d\gamma\asymp  \frac{\d t}{\sqrt{u^2-t}},\,\,\text{for}\,\,t\in(u^2/2,u^2).\end{split}
\ee

Further, 
\begin{equation}\label{8sep20alpha}\begin{split}\sin^2(2\pi\xi_1(t))&=1-\(1-\frac{t}{2}\)^2\asymp t,\,\,\text{for}\,\,t \,\,\text{near}\,\, 0\\
\sin^2(2\pi\xi_2(t))&=1-\(\frac{t-u^2}{2}+1\)^2\asymp u^2-t,\,\,\text{for}\,\,t \,\,\text{near}\,\, u^2.\end{split} 
\ee

We can see from \cref{18nov5alpha,8sep20balpha,9sep20alpha,8sep20alpha} that 
\begin{align*}
\int_{r_m^{-1}((\pm\lambda-\rho)^2)}\hspace{-4pt}|\d\gamma-D^*_\theta \d\gamma| \leq & \int_{r_m^{-1}((\pm\lambda-\rho)^2)} Q_\theta(\xi)\d\gamma\\
\leq &C \theta \( \int_0^{u^2/2}\frac{\d t}{( t+A\theta)\sqrt{t}}+\int_{u^2/2}^{u^2}\frac{\d t}{(u^2-t+B\theta)( u^2-t)^{1/2}}\)\\
\leq &C \theta^{\frac12}\ .
\end{align*}
Here, by an abuse of notation, we are still calling $r_m^{-1}((\pm\lambda-\rho)^2)$ the part of this curve that is in $[0,1/2]\times[0,1/2]$.  The same estimates can be obtained for the  parts  of this curve in the other quadrants.

Using this last estimate on \cref{eq:s2alpha} we obtain
\begin{align*}
||R_3(\lambda)||_1\leq & C (\pm\lambda)^{\frac12} \int_0^{\frac{\pm\lambda-m}{2}}\d\rho \frac{2\lambda^2+2m\lambda\mp2\lambda\rho+\rho^2}{\rho^{\frac12}((\pm\lambda+\rho)^2-m^2)^{\frac12}}\\
\leq & C (\pm\lambda)^{\frac12} \int_0^{\frac{\pm\lambda-m}{2}}\d\rho \frac{2\lambda(\lambda+m)+\tfrac14(\pm\lambda-m)^2}{\rho^{\frac12}(\lambda^2-m^2)^{\frac12}} \\
\leq & C (\pm\lambda)^{\frac12} \(\lambda(\lambda +m)+(\pm\lambda-m)^2 \) \ .
\end{align*}
Thus
\begin{equation}\label{eq:stilde2}
||R_3(\lambda)||_1=\begin{cases}
O(|\lambda+m|) &\text{ if }m \neq 0 , \lambda\uparrow -m\\
O(1) &\text{ if }m \neq 0 , \lambda\downarrow m\\
O(\lambda^2) &\text{ if }m = 0 , \lambda\to 0
\end{cases}\ .
\end{equation}

For $R_4$ we start by noticing
\begin{equation}\label{nablaD}|\nabla r(D_\theta (\xi))|^2=|\nabla r(\xi)|^2- 16\pi^2\theta(-2\cos(2\pi\xi_1)- 2\cos(2\pi\xi_2)+ \theta^2)\ ,\end{equation}
so  we can deduce  
$$\left|\frac{1}{|\nabla r(D_\theta (\xi))|}-\frac{1}{|\nabla r(\xi)|}\right|\leq C \frac{\theta}{|\nabla r (\xi)|^2|\nabla r(D_\theta( \xi))|}.$$

Also, from \cref{le_gradient,18nov6alpha,nabla} we obtain, 
\begin{align*}
||R_4(\lambda)||_1\leq&  C \int_0^{\frac{\pm\lambda-m}{2}}\d\rho \frac{(\pm\lambda)\((\lambda+m)^2+(\pm\lambda-\rho)^2-m^2 \)\((\pm\lambda-\rho)^2-m^2)\)^{\frac12}}{\((\pm\lambda-\rho)^2-m^2)\)\((\pm\lambda+\rho)^2-m^2)\)^{\frac12}}\\
\leq& C\frac{(\pm\lambda)\((\lambda+m)^2+\lambda^2-m^2 \)(\pm\lambda-m)}{(\pm\lambda-m)^{\frac12}(\pm\lambda+3m)^{\frac12}(\pm\lambda-m)^{\frac12}(\pm\lambda+m)^{\frac12}}\\
\leq& C \frac{\pm\lambda^2(\lambda+m)}{(\pm\lambda+3m)^{\frac12}(\pm\lambda+m)^{\frac12}}\ .
\end{align*}
Finally we get
\begin{equation*}
||R_4(\lambda)||_1=\begin{cases}
O(|\lambda+m|) &\text{ if }m \neq 0 , \lambda\uparrow -m\\
O(1) &\text{ if }m \neq 0 , \lambda\downarrow m\\
O(\lambda^2) &\text{ if }m = 0 , \lambda\to 0
\end{cases}\ ,
\end{equation*}
which together with \cref{eq:stilde0,eq:stilde1,eq:stilde2,eq:stildeR2} finishes the proof.
\end{proof}
For the next step notice that 
\begin{equation*}\mathcal{S}_1(\lambda)=\begin{cases}\displaystyle{\int_{\[m^2,M^2_m\]}\d\rho \frac{1}{\rho-\lambda^2}\int_{r_m^{-1}(\rho)}\d \gamma\frac{T_{\lambda,m}(\xi)}{|\nabla r (\xi)|}};&0\leq\pm\lambda<m\\\displaystyle{\int_{\[m^2,\(\frac{m\pm\lambda}{2}\)^2\]\cup\[\(\frac{\pm 3\lambda-m}{2}\)^2,M^2_m\]}\d\rho \frac{1}{\rho-\lambda^2}\int_{r_m^{-1}(\rho)}\d \gamma\frac{T_{\lambda,m}(\xi)}{|\nabla r (\xi)|}};&\pm\lambda>m\end{cases} \ .\end{equation*}

Define $\mathring{T}_{\lambda,m}(\xi)=T_{\lambda,m}(\xi)-T_{\lambda,m}(\xi_0)$. Taking into account \cref{27jul20b,18nov2,18nov} we obtain for $\xi$ small enough
\begin{equation}\label{eq:Tcirculocota}
||\mathring{T}_{\lambda,m}(\xi)||_1\leq C(|m+\lambda||\xi|^\beta+r(\xi)^{\frac12})(|\lambda+m|+r(\xi)^{1/2})\ .
\end{equation}

Set 
\begin{equation*}\mathcal{S}_2(\lambda):=\begin{cases}\displaystyle{\int_{\[m^2,M^2_m\]}\d\rho \frac{1}{\rho-\lambda^2}\int_{r_m^{-1}(\rho)}\d \gamma\frac{\mathring{T}_{\lambda,m}(\xi)}{|\nabla r (\xi)|}};&0\leq\pm\lambda<m\\\displaystyle{\int_{\[m^2,\(\frac{m\pm\lambda}{2}\)^2\]\cup\[\(\frac{\pm 3\lambda-m}{2}\)^2,M^2_m\]}\d\rho \frac{1}{\rho-\lambda^2}\int_{r_m^{-1}(\rho)}\d \gamma\frac{\mathring{T}_{\lambda,m}(\xi)}{|\nabla r (\xi)|}};&\pm\lambda>m\end{cases}\ .\end{equation*} 

Notice that, since $T_{\lambda,m}(\xi_0)$ is of finite rank, for all $r>0$ we have 
\begin{equation}\label{15oct20}
n_\pm(r;{\mathcal{S}}_1(\lambda))=n_\pm(r;{{\mathcal{S}}_2(\lambda)})+O(1).
\end{equation}

 Further, set 
\begin{equation*}\mathcal{S}_{3}(\lambda):=\begin{cases}\displaystyle{\int_{\[m^2,M^2_m\]}\d\rho \frac{1}{\rho-\lambda^2}\int_{r_m^{-1}(\rho)}\d \gamma\frac{\mathring{T}_{\pm m,m}(\xi)}{|\nabla r (\xi)|}};&0\leq\pm\lambda<m\\\displaystyle{\int_{\[m^2,\(\frac{m\pm\lambda}{2}\)^2\]\cup\[\(\frac{\pm 3\lambda-m}{2}\)^2,M^2_m\]}\d\rho \frac{1}{\rho-\lambda^2}\int_{r_m^{-1}(\rho)}\d \gamma\frac{\mathring{T}_{\pm m,m}(\xi)}{|\nabla r (\xi)|}};&\pm\lambda>m\end{cases}\ .\end{equation*} 

\begin{Lemma}\label{le:tcirculocotabis}
For $\lambda$ near $\pm m$ and $\xi$ small enough we have
\begin{equation*}
||\mathring{T}_{\lambda,m}(\xi)-\mathring{T}_{\pm m,m}(\xi)||_1\leq C(|\lambda\mp m| r(\xi)^\frac{\beta_1}2)(|\pm m+m|+|\lambda+m|+r(\xi)^{\frac12} )\ .
\end{equation*}

\end{Lemma}
\begin{proof}
Let us start by defining 
\begin{equation*}
s^\pm_{\lambda,m}(\xi):=t_{\lambda,m}(\xi)-t_{\lambda,m}(\xi_0)-t_{\pm m,m}(\xi)+t_{\pm m,m}(\xi_0)\ .
\end{equation*}
One can readily see that $s^\pm_{\lambda,m}(\xi):L^2(\T^2;\C^3)\to\C^3$ is given by 
$$
\begin{pmatrix}f_1\\f_2\\ f_3\end{pmatrix}\mapsto
\begin{pmatrix}
(\lambda\mp m)\displaystyle{\int_{\T^2}\d\eta\ (\hat{g}_1(\xi-\eta)-\hat{g}_1(-\eta))f_1(\eta)}\\[.7em]
0
\\[.7em]
0
\end{pmatrix}\ ,
$$
from which together with \cref{16oct20} we obtain
\begin{equation}\label{eq:schicocota}
||s^\pm_{\lambda,m}(\xi)||_2\leq|\lambda\mp m|\,|\xi|^{\beta_1}\ .
\end{equation}
Noticing 
\begin{align}
&\hspace{-25pt}\mathring{T}_{\lambda,m}(\xi)-\mathring{T}_{\pm m,m}(\xi)\nonumber\\
=&(s^\pm_{\lambda,m}(\xi)^*-t_{\pm m,m}(\xi_0)^*)At_{\lambda,m}(\xi) + t_{\lambda,m}(\xi_0)^*A(t_{\pm m,m}(\xi)+s^\pm_{\lambda,m}(\xi))\label{eq:Tcirculodif}  \\
&+ t_{\pm m,m}(\xi)^*A(s^\pm_{\lambda,m}(\xi)+t_{\lambda,m}(\xi_0)) +(s^\pm_{\lambda,m}(\xi)^*-t_{\lambda,m}(\xi)^*)At_{\pm m,m}(\xi_0)\nonumber\\
=&s^\pm_{\lambda,m}(\xi)^*At_{\lambda,m}(\xi) + t_{\lambda,m}(\xi_0)^*As^\pm_{\lambda,m}(\xi)+t_{\pm m,m}(\xi)^*As^\pm_{\lambda,m}(\xi)+s^\pm_{\lambda,m}(\xi))At_{\pm m,m}(\xi_0)\nonumber\ ,
\end{align}
one concludes taking into account \cref{eq:schicocota,27jul20b}.
\end{proof}

From this lemma one can easily see that 
 \begin{equation}\label{2}\\
||{\mathcal{S}}_2(\lambda)-{\mathcal{S}}_3(\lambda)||_1=
O(|\lambda+m|),  \quad \lambda\rightarrow \pm m \ .
\end{equation}

Using  \cref{eq:Tcirculocota} we  can deduce that there exists $\rho_0$ such that for all $m^2<\rho\leq \rho_0$ 
\begin{equation}\label{21nov}
\int_{r^{-1}_m(\rho)}\d\gamma\frac{\left|\left| \mathring{T}_{\pm m,m}(\xi)\right|\right|_1}{|\nabla r (\xi)|}  \leq C \(|\pm m+m|^2(\rho-m^2)^{\frac{\beta_1}2}+|\pm m+m|(\rho-m^2)^{\frac12}+(\rho-m^2)\)\ ,
\end{equation}

thus
\begin{equation*}
{\mathcal{S}}_3(\pm m):=\int_{m^2}^{M_m^2} \frac{\d\rho}{\rho-m^2} \int_{r^{-1}_m(\rho)}\d\gamma\frac{\mathring{T}_{\pm m,m}(\xi)}{|\nabla r (\xi)|} 
\end{equation*}
is well defined in the trace class. 


\begin{Lemma}\label{le:S10menosalpha}
For $m>0$ 
 \begin{equation}\label{1} 
||{\mathcal{S}}_{3}(\lambda)-{\mathcal{S}}_3(-m)||_1=O(|\lambda+m|\,|\ln(|\lambda+m|)|)\quad  \lambda\to -m\ . 
\end{equation}
\begin{equation}
||{\mathcal{S}}_{3}(\lambda)-{\mathcal{S}}_3(0)||_1=
O(|\lambda|),\quad \lambda\rightarrow 0\ .
\end{equation}
\begin{equation}
||{\mathcal{S}}_{3}(\lambda)-{\mathcal{S}}_3(m)||_1=
O(1),\quad \lambda\rightarrow m\ .
\end{equation}
\end{Lemma}
\begin{proof} We only give the proof of \cref{1} as the others use  similar ideas. 
Let us fix $\rho_0$ such that \cref{21nov} holds. Without loss of generality suppose $\(\frac{3\lambda+m}{2}\)^2<\rho_0$. We start by comparing 
\begin{equation*}
\left|\left|
\int_{\rho_0}^{M_m^2}\d\rho\(\frac{1}{\rho-\lambda^2}\int_{r^{-1}_m(\rho)}\d\gamma\frac{\mathring{T}_{\lambda,m}(\xi )}{|\nabla r (\xi)|}-\frac{1}{\rho-m^2}\int_{r^{-1}_m(\rho)}\d\gamma\frac{\mathring{T}_{-m,m}(\xi)}{|\nabla r (\xi )|}\) 
\right|\right|_1 \ .
\end{equation*}
By \cref{le:tcirculocotabis} we see that it is bounded by 
\begin{equation*}\begin{split}
&C|\lambda+m| 
\int_{\rho_0}^{M_m^2}\d\rho\(\frac{\lambda-m}{(\rho-\lambda^2)(\rho-m^2)}\int_{r^{-1}_m(\rho)}\d\gamma\frac{||\mathring{T}_{\lambda,m}(\xi )||_1}{|\nabla r (\xi)|}+\frac{1}{\rho-m^2}\int_{r^{-1}_m(\rho)}\frac{\d\gamma}{|\nabla r (\xi )|}\)\\
\leq & C|\lambda+m|\ , \end{split}
\end{equation*}
where the last inequality follows by recalling the proof of \cref{le_gradient}.
Next, for $\lambda<-m<0$, the remainder term is given by
\begin{align}
&\int_{\[m^2,\(\frac{m-\lambda}{2}\)^2\]\cup\[\(\frac{3\lambda+m}{2}\)^2,\rho_0\]}\d\rho\left(\frac{1}{\rho-\lambda^2}-\frac{1}{\rho-m^2} \right)\int_{r^{-1}_{m}(\rho)}\d\gamma\frac{\mathring{T}_{\lambda,m}(\xi)}{|\nabla r (\xi)|}\nonumber \\
&-\int_{\(\frac{m-\lambda}{2}\)^2}^{\(\frac{-3\lambda-m}{2}\)^2} \frac{\d\rho}{\rho-m^2} \int_{r^{-1}_{m}(\rho)}\d\gamma\frac{\mathring{T}_{\lambda,m}(\xi)}{|\nabla r (\xi)|}\label{eq:remainder}\\
& +\int_{m^2}^{\rho_0}\frac{\d\rho}{\rho-m^2}\int_{r^{-1}_{m}(\rho)}\d\gamma\frac{\mathring{T}_{\lambda,m}(\xi)-\mathring{T}_{-m,m}(\xi)}{|\nabla r (\xi)|}\nonumber \ .
\end{align}
Notice that for $-m<\lambda<0$ the second term vanishes and the first integral is over $\[m^2,\rho_0\]$.
 First we see from \cref{eq:Tcirculocota} that, 
\begin{align*}
\left|\left|\int_{r^{-1}_m(\rho)}\d\gamma\frac{\mathring{T}_{\lambda,m}(\xi )}{|\nabla r (\xi)|}\right|\right|_1
 & \leq C \int_{r^{-1}_m(\rho)}\d\gamma \frac{|\lambda+m|^2r(\xi)^{\frac{\beta_1}2}+|\lambda+m|r(\xi)^{\frac{1}2}+r(\xi) }{r (\xi )^{\frac12}}\\
 & \leq C\(|\lambda+m|^2\,|\rho-m^2|^{\frac{\beta_1}2}+|\lambda+m|\,|\rho-m^2|^{\frac{1}2}+|\rho-m^2|\)\ .
\end{align*}
To estimate the first term of \cref{eq:remainder} for $\lambda\uparrow-m$ , we use this last inequality  and that  for $\lambda \in \[m^2,\(\frac{m-\lambda}{2}\)^2\]\cup\[\(\frac{3\lambda+m}{2}\)^2,\rho_0\]$  we have that $|\lambda^2-\rho|\geq C |\lambda+m|$.  So
\begin{align*}
\int_{m^2}^{\(\frac{m-\lambda}2\)^2}\d\rho \frac{|\lambda^2-m^2|}{|\rho-\lambda^2|\,|\rho-m^2|}\(|\lambda+m|^2\,|\rho-m^2|^{\frac{\beta_1}2}+|\lambda+m|\,|\rho-m^2|^{\frac{1}2}+|\rho-m^2|\)& \\
&\hspace{-275pt}\leq C \int_{m^2}^{\(\frac{m-\lambda}2\)^2}\d\rho \(|\lambda+m|^2 |\rho-m^2|^{\frac{\beta_1}2-1}+|\lambda+m|\,|\rho-m^2|^{\frac{1}2-1}+1\) \\
& \hspace{-275pt}\leq C|\lambda+m|\ ,
\end{align*}
and
\begin{align*}
\int_{\(\frac{m+3\lambda}2\)^2}^{\rho_0}\frac{|\lambda^2-m^2|}{|\rho-\lambda^2|\,|\rho-m^2|}\(|\lambda+m|^2\,|\rho-m^2|^{\frac{\beta_1}2}+|\lambda+m|\,|\rho-m^2|^{\frac{1}2}+|\rho-m^2|\)&\\
& \hspace{-175pt}\leq C\(|\lambda+m|^2+|\lambda+m|+\int_{\(\frac{m+3\lambda}2\)^2}^{\rho_0}\frac{|\lambda^2-m^2|}{|\rho-\lambda^2|}  \)\\
& \hspace{-175pt}\leq C(|\lambda+m|\,|\ln(|\lambda+m|)|) \ .
\end{align*}

Finally, the previous computation also gives us for $-m<\lambda<0$ (and hence $0<\lambda^2<m^2<\rho$) 
\begin{align*}
\int_{m^2}^{\rho_0}\frac{|\lambda^2-m^2|}{|\rho-\lambda^2|\,|\rho-m^2|}\(|\lambda+m|^2\,|\rho-m^2|^{\frac{\beta_1}2}+|\lambda+m|\,|\rho-m^2|^{\frac{1}2}+|\rho-m^2|\)&\\
& \hspace{-90pt}\leq C(|\lambda+m|\,|\ln(|\lambda+m|)|) \ . \qedhere
\end{align*}

\end{proof}

\section{Proof of the main results for hyperbolic thresholds }\label{sec:hyper}

We turn now our attention to the \emph{hyperbolic thresholds}. For ease of notation we will set $m=0$ and consider the positive hyperbolic threshold $\tau=2$. 

 First,  taking into account \cref{21jan20_a},  we will need to study the operator
\begin{equation*}
\mathcal{R}(\lambda):=\text{p.v.}\int_0^{M_0}\d u \frac{2u}{u^2-\lambda^2} \int_{r^{-1}(u^2)}\d\gamma\frac{{T}_\lambda(\xi)}{|\nabla r (\xi)|} .
\end{equation*}
Since $V$ is trace class,  for all $r>0$ 
\begin{equation}\label{15oct20b}n_\pm(r;{\rm Re}(K(\lambda+i0)))=n_\pm(r ;\frac1\lambda\mathcal{R}(\lambda))+O(1).\ee

When integrating on $\T^2$, due to the periodicity, we will focus on the  rectangle $[0,1/2]\times [0,1/2]$. 
Let us take the critical points $\xi_a=(1/2,0)$ and $\xi_b=(0,1/2)$ in $\T^2$. Notice that the level curve defined by the threshold $\tau=2$, namely $r^{-1}(4)$, pass through these points.   

Define  the triangles $\triangle_a:=\{\xi\in[0,1/2]\times [0,1/2]: \xi_2\leq \xi_1\}$ and $\triangle_b:=\{\xi\in[0,1/2]\times [0,1/2]: \xi_2> \xi_1\}$
 and set 
\begin{equation*}
\mathring{T}_\lambda(\xi):=\mathds{1}_{\triangle_a}(\xi)(T_\lambda(\xi)-T_\lambda(\xi_a))+\mathds{1}_{\triangle_b}(\xi)(T_\lambda(\xi)-T_\lambda(\xi_b))\ .
\end{equation*}
For simplicity, all the following computations will be performed on the triangle $\triangle_a$. We will also assume that  $\lambda\geq 2$. 

Denote by $\gamma_{\lambda}$ the part of the curve  
$r^{-1}(\lambda^2)$ that lies in $\triangle_a$. It admits the following parametrization:  
\begin{equation}\label{paramteri} \xi_1(t)=\frac{1}{2\pi} \arccos\left(1-\frac{t}{2}\right);\quad \xi_2(t)=\frac{1}{2\pi} \arccos\left(1+\frac{t-\lambda^2}{2}\right)\ ,
\end{equation}
for $t\in[\tfrac{\lambda^2}2, 4]$.   We can easily compute 
$$\xi_1'(t)=\frac{1}{2\pi t^{1/2}(4-t)^{1/2}};\quad \xi_2'(t)=-\frac{1}{2\pi (\lambda^2-t)^{1/2}(4+t-\lambda^2)^{1/2}}\ ,$$
so 
\begin{equation}\label{dgammat4}
\d\gamma\leq C( (4-t)^{-1/2})\d t\ . 
\end{equation}
Most computations will make use of this explicit parametrization, but other parts of $\T^2$ can be dealt with accordingly, as can be the case $\lambda\leq 2$.

\begin{Lemma}\label{imaginaryhyperbolic}
The following asymptotic relation holds true
$$\left|\left|\int_{\gamma_\lambda}\d\gamma\frac{\mathring{T}_\lambda(\xi)}{|\nabla r (\xi)|}\right|\right|_1=O(1), \quad \lambda\to \tau.$$
\end{Lemma}
\begin{proof}
 For $\xi\in\triangle_a$, from \cref{18nov,18nov2} we can see that 
 \begin{equation}\label{Tcirculocotahyper}
 ||\mathring{T}(\xi)||_1\leq C |\xi-\xi_a|^{\beta}\ .
 \end{equation} Further, we have that $|\nabla r(\xi)|^2\asymp r(\xi-\xi_a)\asymp |\xi-\xi_a|^2$ so 
\begin{equation}\label{26aug20}\left|\left|\frac{\mathring{T}_\lambda(\xi)}{|\nabla r (\xi)|}\right|\right|_1 \leq C|\xi-\xi_a|^{\beta-1}\ ,\ee
when 
$|\lambda-\tau|$ is  small enough. 
Therefore 
\begin{equation*}
\left|\left|\int_{\gamma_\lambda}\d\gamma\frac{\mathring{T}_\lambda(\xi)}{|\nabla r (\xi)|}\right|\right|_1\leq  C\int_{\gamma_\lambda }\d\gamma|\xi-\xi_a|^{\beta-1}\ .
\end{equation*}
A direct inspection gives us that for $t\in[\tfrac{\lambda^2}2, 4]$
\begin{equation}\label{cotaparametrizada}
|\xi(t)-\xi_a|\leq C (4-t)^{\frac12}\ .
\end{equation}
From this we obtain
\begin{equation*}
\int_{\gamma_\lambda}\d\gamma|\xi-\xi_a|^{\beta-1}\leq C \int_{\tfrac{\lambda^2}2}^4\d t\frac{((4-t)^{\frac12})^{\beta-1}}{(4-t)^{\frac12}}\leq C\int_{\tfrac{\lambda^2}2}^4\d t (4-t)^{\frac\beta2-1}=O(1) \text{ as } \lambda\to \tau\ . \qedhere
\end{equation*}
\end{proof}

Set 
\begin{equation}\label{eq:Slambdabis}
\mathcal{R}_1(\lambda):=\text{p.v.}\int_0^{M_0}\d u \frac{2u}{u^2-\lambda^2} \int_{r^{-1}(u^2)}\d\gamma\frac{\mathring{T}_\lambda(\xi)}{|\nabla r (\xi)|} .
\end{equation}

Denote by $\gamma_{\lambda,b}$ the part of the curve $r^{-1}(\lambda^2)$ that lies in $\triangle_b$, then  $$\int_{r^{-1}(\lambda^2)}\d\gamma\frac{ {{T}}_\lambda(\xi)-\mathring{T}_\lambda(\xi)}{|\nabla r (\gamma)|}={{T}}_\lambda(\xi_a)\int_{\gamma_\lambda }\d\gamma\frac{1}{|\nabla r (\gamma)|}+{{T}}_\lambda(\xi_b)\int_{\gamma_{\lambda,b} }\d\gamma\frac{1}{|\nabla r (\gamma)|}$$ is an operator with  rank at most 
 twelve. Thus, we can use \cref{lepushbis} together with \bk \cref{21jan20,15oct20b,imaginaryhyperbolic} to obtain that 
for any $\epsilon \in (0,1)$
 \begin{equation}\label{5aug20c} 
\begin{split}
\pm n_\mp\left((1\pm\epsilon);\frac{1}{\lambda}\mathcal{R}_1(\lambda)\right)+O(1)\leq& \xib(\lambda;H_\pm,H_0)\\
 \leq& \pm n_\mp\left((1\mp\epsilon);\frac{1}{\lambda}\mathcal{R}_1(\lambda)\right)+O(1) \ .
\end{split}
 \end{equation}
  as $\lambda\to \tau$.

\begin{Lemma}\label{le_holder} There exists $u_0>0$ such that for $u, u'$ in $(2-u_0,2]$ or  in  $[2,2+u_0)$ 
$$ \left|\left|\int_{r^{-1}(u'^2)}\d\gamma\frac{\mathring{T}_\lambda(\xi)}{|\nabla r (\xi)|} -\int_{r^{-1}(u^2)}\d\gamma\frac{\mathring{T}_\lambda(\xi)}{|\nabla r (\xi)|}\right|\right|_1= O(|u-u'|^{\beta \bk /2} \ln|u-u'|^{-1}),\quad |u-u'|\to0.$$
\end{Lemma}
\begin{proof} Assume that $\sqrt{8}>u'>u\geq\tau=2$. As before we consider the triangle $\triangle_a$. 
Recall that we denote  by $\gamma_{u}$ the part of the curve
$r^{-1}(u^2)$  that lies in $\triangle_a$.   

 Set $\zeta:=(u'^2-u^2)/2$ and 
define 
 the function 
 $$D_\zeta(\xi_1,\xi_2)=\(\frac1{2\pi}\arccos(\zeta+\cos(2\pi\xi_1)),\xi_2\)\ .$$
Notice that $D_\zeta(\gamma_u')\subset \gamma_u$ and 
$$\int_{\gamma_{u}}\d\gamma\frac{\mathring{T}_\lambda(\xi)}{|\nabla r (\xi)|}=\int_{\gamma_u'}D_\zeta^*\d{\gamma}\frac{\mathring{T}_\lambda(D_\zeta\xi)}{|\nabla r (D_\zeta\xi)|}+\int_{\gamma_u\setminus D_\zeta\gamma_u'}\d{\gamma}\frac{\mathring{T}_\lambda(\xi)}{|\nabla r (\xi)|}\ , $$ so we need to estimate 
\bel{13may21}\int_{\gamma_{u'}}(D_\zeta^*\d\gamma-\d\gamma)\frac{\mathring{T}_\lambda(D_\zeta\xi)}{|\nabla r (D_\zeta\xi)|}+\int_{\gamma_{u'}}\d{\gamma}\frac{\mathring{T}_\lambda(D_\zeta\xi)-\mathring{T}_\lambda(\xi)}{|\nabla r (D_\zeta\xi)|}+\int_{\gamma_{u'}}\d{\gamma}\,\mathring{T}_\lambda(\xi)\left(\frac{1}{|\nabla r (D_\zeta\xi)|}-\frac{1}{|\nabla r (\xi)|}\right)\ee
$$:=(I)+(II)+(III)\ ,$$
and 
$$
\int_{\gamma_u\setminus D_\zeta\gamma_{u'}}\d{\gamma}\frac{\mathring{T}_\lambda(\xi)}{|\nabla r (\xi)|}:=(IV) \ .
$$

We start by estimating (IV).  It is easy to see that  $\gamma_u\setminus D_\zeta\gamma_u'$  can be parametrized by  \cref{paramteri} with $t\in[4+u^2-u'^2,4]$.  Then, by \cref{dgammat4,26aug20,cotaparametrizada}
\bel{12may21}
\left|\left|\int_{\gamma_u\setminus D_\zeta\gamma_{u'}}\d{\gamma}\frac{\mathring{T}_\lambda(\xi)}{|\nabla r (\xi)|}\right|\right|_1\leq C\int_{4+u^2-u'^2}^4\d t (4-t)^{\beta/2-1}\leq C |u'-u|^{\beta/2}\\\
\ee

To estimate $(I)+(II)+(III)$ a partition of $\triangle_{a}$ is needed, namely   $\triangle_{a,1}:=\{\xi\in\triangle_a: \xi_1\geq 3/8,           \xi_2\leq1/8\}$  and $\triangle_{a,2}:=\triangle_{a}\setminus\triangle_{a,1}$.  Accordingly we write $\gamma_{u,j}$ for the part of the curve
$r^{-1}(u^2)$   that lies in $\triangle_{a,j}$, $j=1,2$. The same for $(I_j), (II_j),(III_j)$.

Let us start by studying $(I_1), (II_1),(III_1)$. Consider first $(I)$ on   $\triangle_{a,1}$. We have that $|\d\gamma-D_\zeta^*\d\gamma|\leq |\partial_1D_{\zeta,1}(\gamma)^2-1|\d\gamma$, and 
\begin{equation}\label{30sep20}
|\partial_1 D_{\zeta,1}(\xi)^2-1|=\frac{|\zeta(2\cos2\pi\xi_1+\zeta)|}{|\sin^22\pi\xi_1-\zeta(2\cos2\pi\xi_1+\zeta)|}\ .
\ee
On $\triangle_{a,1}$ we have $-\zeta(2\cos2\pi\xi_1+\zeta)>\tilde{C}\zeta$, for a positive $\tilde{C}$ (when $\zeta$ is sufficiently small). Putting all this together with \ref{26aug20}
\begin{equation}\label{30sep20b}||(I_1)||_1\leq C\int_{\gamma_{u',1}}|\d\gamma-D_\zeta^*\d\gamma| |\xi-\xi_a|^{\beta-1}
\leq C \zeta\int_{\gamma_{u',1}}\d\gamma \frac{|\xi-\xi_a|^{\beta-1}}{|\sin^22\pi\xi_1+\tilde{C}\zeta|}, \ee

and since 
$$ \sin^22\pi\xi_1(t)=\frac{t(4-t)}{4},$$
by \cref{dgammat4,cotaparametrizada}
$$||(I_1)||_1\leq C\zeta \int_{u'^2/2}^{4}\d t\frac{(4-t)^{\beta/2-1}}{4-t+\tilde{C}\zeta}\leq C\zeta^{\beta/2} \int_{0}^{\infty}\d s\frac{s^{\beta/2-1}}{s+\tilde{C}}\  .$$

Now, for $(II_1)$, proceeding as in \cref{18nov2aalpha}
we can show that
\begin{equation}\label{22jul20} 
 ||\mathring{T}_\lambda(D_\zeta\xi)-\mathring{T}_\lambda(\xi)||_1\leq C    |\xi-D_\zeta\xi|^{\beta}\leq C |\zeta|^{\frac{\beta}{2}}\ .
 \end{equation}
Further 
\begin{equation}\label{1oct20b}|\nabla r(D_\zeta\xi)|^2=|\nabla r(\xi)|^2-16\pi \zeta(\zeta+2\cos 2\pi\xi_1)
,\ee
and using the same parametrization as before
\begin{equation}\label{12may21b}\begin{split}|\nabla r(D_\zeta\xi(t))|^2&= 4\pi^2t(4-t)+4\pi^2(4-u^2+t)(u^2-t)-16\pi^2\zeta(\zeta+(2-t))\\&\geq \tilde{B} (4-t)+\tilde{C}\zeta,\end{split}\end{equation}
for some $\tilde{B}, \tilde{C}>0$.
Then, 
\begin{align*}
||(II_1)||_1\leq&\left|\left|\int_{\gamma_{u',1}}\d{\gamma}\frac{\mathring{T}_\lambda(D_\zeta\xi)-\mathring{T}_\lambda(\xi)}{|\nabla r (D_\zeta\xi)|}\right|\right|_1
\leq  C|\zeta|^{ \beta /2}\int_{\gamma_{u',1}}\frac{\d\gamma}{|\nabla r( D_\zeta\xi)|}\\
\leq &C |\zeta|^{ \beta /2}  \int_{u'^2/2}^{4}\d t\frac{(4-t)^{-1/2}}{(4-t+{\tilde{C}}\zeta)^{1/2}} 
\leq C \zeta^{ \beta \bk/2}  \ln(\zeta^{-1})\ .
\end{align*}
Finally consider $(III_1)$. Since 
$$\left| \frac{1}{|\nabla r (\xi)|}-\frac{1}{|\nabla r (D_\zeta\xi)|}\right|=\frac{\Big||\nabla r (D_\zeta\xi)|^2-|\nabla r(\xi)|^2\Big|}{|\nabla r (\xi)|\,|\nabla r (D_\zeta\xi)|(|\nabla r (\xi)|+|\nabla r (D\xi)|)},$$  from \cref{26aug20,1oct20b,12may21b}
$$||(III_1)||_1
\leq C |\zeta| \int_{\gamma_{u'^2}}\d \gamma\frac{|\xi-\xi_a|^{\beta-1}}{|\nabla r(D\gamma)|^2}\\ \leq  \zeta \int_{u'^2/2}^{4} \d t\frac{(4-t)^{\beta/2-1}}{B(4-t)+\tilde{A}\zeta}\\\leq C \zeta^{\beta/2}\ .$$

Similar arguments can be used to deal with  $(I_2), (II_2),(III_2)$.  In fact, for  these cases  the computations are simpler because if $u_0$ is small enough  all the functions appearing on the integrals  of \cref{13may21} are H\"older continuous (see \cref{le_T}). Moreover,  the denominators are non-vanishing, so the result follows immediately when we restrict to $\triangle_{a,2}$. \end{proof}

For $\lambda>2$  we write principal value part of the operator $\mathcal{R}_1(\lambda)$  (defined in \cref{eq:Slambdabis}) as follows
\begin{align}
\text{p.v.}& \int_{0}^{M_0}\hspace{-10pt} \d u\frac{2 u}{(u^2-\lambda^2)}\int_{r^{-1}(u^2)} \d \gamma  \frac{\mathring{T}_\lambda(\xi)}{|\nabla r({\xi})|}\nonumber \\
&= \int_{[0,(\lambda+2)/2]\cup[(3\lambda-2)/2,M_0]}\hspace*{-10pt} \d u\frac{2 u}{(u^2-\lambda^2)}\int_{r^{-1}(u^2)}\d \gamma \frac{\mathring{T}_\lambda(\xi)}{|\nabla r (\xi)|} \label{5oct20}  \\
&+\int_0^{(\lambda-2)/2}  \frac{\d \rho }{\rho}\left(\frac{2 (\lambda+\rho)}{(2\lambda+\rho)}\int_{r^{-1}((\lambda+\rho)^2)}\d\gamma\frac{\mathring{T}_\lambda(\xi)}{|\nabla r (\xi)|} -\frac{2 (\lambda-\rho)}{(2\lambda-\rho)}\int_{r^{-1}((\lambda-\rho)^2)}\d\gamma\frac{\mathring{T}_\lambda(\xi)}{|\nabla r (\xi)|}\right)\nonumber \\
&=:\tilde{\mathcal{R}}_{1}(\lambda)+{\mathcal{R}}_{1,pv}(\lambda)\ . \nonumber
\end{align}

Thus from \cref{le_holder}   we can readily see that 
\begin{equation}\label{1oct20}
||{\mathcal{R}}_{1,pv}(\lambda)||_1=o(|2-\lambda|),\quad \lambda\to2\ .
\ee

Set  $$ \mathcal{R}_2(\lambda)  :=\int_{[0,(\lambda+2)/2]\cup[(3\lambda-2)/2,M_0]}\hspace*{-10pt} \d u\frac{2 u}{(u^2-\lambda^2)}\int_{r^{-1}( \lambda^2)}\d \gamma \frac{\mathring{T}_\lambda(\xi)}{|\nabla r (\xi)|}.$$
From \cref{le_holder}   we can see that 
\begin{equation}\label{1oct20d}
||\tilde{\mathcal{R}}_1(\lambda)-\mathcal{R}_2(\lambda)||_1=O(1).
\end{equation}
Finally,  the function 
$$\lambda\mapsto\int_{[0,(\lambda+2)/2]\cup[(3\lambda-2)/2,M_0]}\hspace*{-10pt} \d u\frac{2 u}{(u^2-\lambda^2)}=\ln\left(\frac{(3\lambda+2)(M_0^2-\lambda^2)}{\lambda^2(5\lambda-2)}\right)$$ is bounded for $\lambda\to\tau=2$, so $\mathcal{R}_2(\tau)$ is well defined and compact. Moreover, we notice that for  $\xi\in\mathcal{U}_a$
\begin{equation}\label{10nov20}\begin{split}
||\mathring{T}_\lambda(\xi)-\mathring{T}_2(\xi)||_1
=&||[(t_\lambda(\xi)-t_\lambda(\xi_a))^*-(t_2(\xi)-t_2(\xi_a))^*]At_\lambda(\xi)\\
&+ [(t_2(\xi)-t_2(\xi_a))^*]A(t_\lambda(\xi)-t_2(\xi))\\
&+ [(t_\lambda(\xi_a)-t_2(\xi_a))^*]A(t_2(\xi)-t_\lambda(\xi_a))\\
&+t_2(\xi_2)^*A[(t_\lambda(\xi)-t_\lambda(\xi_2))-(t_2(\xi)-t_2(\xi_a))]||_1\\
\leq & C |\lambda-2||\xi-\xi_2|^{\beta}\ . 
\end{split}
\ee
An analogous computations holds when $\xi\in\mathcal{U}_b$. From \cref{10nov20}, together with \cref{le_holder} we obtain that 
\begin{equation}\label{11nov20}
\mathcal{R}_1(\lambda)\to\mathcal{R}_1(\tau),\quad \lambda\to\tau=2\ ,
\ee
in the trace class norm.

In consequence, putting together with  \cref{5aug20c,5oct20,1oct20,1oct20b,11nov20} we conclude that
if $V$ satisfies \cref{3dec19_3} then 
$$\xib(\lambda;H_\pm,H_0)=O(1), \quad \lambda \to2.$$

\section{Eigenvalue Asymptotics for integral operators with toroidal kernel}\label{Raikov}

Denote by  $S_\rho^\gamma(\Z^d)$  the class of symbols given by the functions $v:\Z^d\to \C$ that satisfies {for any multi-index} $\alpha$
\bel{31may21_10} |{\rm D}^\alpha v(\mu)|\leq C_\alpha \langle \mu \rangle^{-\gamma-\rho|\alpha|},\ee
where ${\rm D_{\mu_j}} v(\mu):=v(\mu+\delta_j)-v(\mu)$, and ${\rm D}^\alpha:={\rm D}_{\mu_1}^{\alpha_1}...{\rm D}_{\mu_d}^{\alpha_d} $.

\emph{Condition 1: For  $\gamma>d$, $\rho>0$   assume $\{v_k\}_{k=1}^N\in S_\rho^\gamma(\Z^d)$.  We suppose also that    
 \bel{3jun21}v_k(\mu)=v_0(\mu)(\Gamma_k +o(1)),\quad|\mu|\to \infty,\ee for a function 
$v_0:\Z^d\to\C $ which, viewed as a multiplication operator, satisfies  
\bel{2jun21} n_\pm(\lambda; v_0)=\lambda^{-d/\gamma}(c_\pm+o(1)), \quad \lambda\downarrow0,\ee
with  $c_\pm>0$. }

Let $\{B_k\}_{k=1}^N$ be a family of functions in $L^2(\T^d)$.  Define  $$\Psi:=\sum_{k=1}^NB_k\F{v_k}\F^*\overline{B}_k\ .$$ This is a compact  operator  on $L^2(\T^d)$ and has  integral  kernel given by \begin{equation}\label{eq:kernelefectivo}
\sum_{k=1}^NB_k(\xi)\widehat{v_k}(\xi-\eta)\overline{B}_k(\eta)\ .
\end{equation}

\begin{Theorem}\label{Th_raikov} 
Assume Condition 1 and $\Psi$ as above. Then 
$$ n_\pm(\lambda;\Psi)=\mathcal{C}_{B\pm}\lambda^{-d/\gamma}(1+o(1)),\quad \lambda\downarrow0,$$
where $$\mathcal{C}_{B\pm}=c_\pm\int_{\T^d}\d\xi\Big(\sum_{k=1}^N\Gamma_k|B_k(\xi)|^2\Big)^{d/\gamma}\ .$$ 
\end{Theorem}

\begin{Remark}
Using the ideas of \cite{Birman-Solomyak-70}, \cref{Th_raikov} can be easily extended to cover more general kernels. In particular, in \cref{eq:kernelefectivo} one could consider matrix-valued $B$ or introduce another hermitian matrix-valued function depending on $\xi$ and $\eta$. Here we limit ourselves to setting we were aiming to apply it. Although \cref{Th_raikov} bears resemblance with \cite[Theorem 1.]{Birman-Solomyak-70}, we would like to stress out that it does not follows directly from it or it's related results (see for instance \cite{MR0463738,MR1306510}). 
 \end{Remark}

Set $\blacksquare:=[0,1)^d$. In order to prove the following lemma we need to use the spaces of compact operators $\mathfrak{S}_{p,w}$  defined for $0<p<\infty$ by 
$$\mathfrak{S}_{p,w}:\{K\in \mathfrak{S}_\infty: s_n(K)=O(n^{-1/p})\}$$
with the quasi-norm 
\begin{equation}\label{wti}||K||_{p,w}:=\sup_n\{n^{1/p}s_n(K)\}=\left(\sup_{s>0} \{s^pn_*(s;K)\}\right)^{1/p}\ .
\end{equation}
These spaces satisfy  the "weakened triangle inequality" $||K_1+K_2||_{p,w}\leq2^{1/p}(||K||_{p,w}+||K||_{p,w})$ and the "weakened H\"older inequality" 
\bel{weak_holder} ||K_1K_2||_{r,w}\leq c(p,q) ||K_1||_{p,w}||K_2||_{q,w},\ee
for $r^{-1}=p^{-1}+q^{-1}$ and $c(p,q)=(p/r)^{1/p}(q/r)^{1/q}$
(see [Chapter 11]\cite{BS87}).

In order to prove \Cref{Th_raikov} we need the following Cwikel--type estimate. 

\begin{Lemma}\label{Cwikel-Bir-Sol} Assume that $|v(\mu)|\leq C \langle \mu\rangle^{-\gamma}$.  Then,   if $p,q\in L^2(\blacksquare)$,  there exists a positive constant $C(v,d)$ such that 
$$  ||p\F v\F^* q||_{d/\gamma,w} \leq C(v,d) ||p||_{L^2(\blacksquare)} ||q||_{L^2(\blacksquare)}$$
\end{Lemma}
\begin{proof} By \cref{weak_holder}
\bel{29may21}  ||p\F v\F^* q||_{d/\gamma,w} \leq ||p\F|v|^{1/2}||_{2d/\gamma,w} |||v|^{1/2}\F^* q||_{2d/\gamma,w},\ee
Then it is enough to prove $||p\F|v|^{1/2}||_{2d/\gamma,w}\leq C ||p||_{L^2(\blacksquare)}$. First notice that  by the min-max principle the singular values satisfies 
\bel{29may21b}s_n(p\F|v|^{1/2})^2=\lambda_n(p\F|v|\F^* \overline{p})\leq C \lambda_n(p\F w\F^* \overline{p}),\ee
where $w(\mu)=\langle \mu\rangle^{-\gamma}$, and $\{\lambda_n(K)\}$ denotes the non increasing sequence of eigenvalues of $0\leq K \in \mathfrak{S}_\infty$.

Set $g(x):=\langle x\rangle^{-\gamma}$. Since 
$\gamma>d$, we have that  $g\in L^1(\R^d)$ and  $\hat{g}\in C_0(\R^d)$.  Here we use the standard notation $\hat{g}$ for 
\begin{equation}
\hat{g}(y):=\int_{\R^d}\d x\, \e^{-2 \pi i y\cdot x}g(x) \ .
\end{equation}Further,  $\hat{g}$ is smooth in $\R^d\setminus\{0\}$ and  decay at infinity faster than $|y|^{-n}$  for any $n\in\Z_+$ (same for its derivatives). 
Then, we can use  the Poisson summation formula (see for instance \cite{Grafakos}) to obtain 
\begin{align}\sum_{\mu\in\Z^d}\omega(\mu)e^{-2\pi i \mu\cdot x}&=\sum_{\mu\in\Z^d\setminus{0}}\check{\hat{g}}(\mu)e^{-2\pi i \mu\cdot x}+\omega(0) \nonumber \\
&=\sum_{\mu\in\Z^d}\check{{g}}(-x +\mu)-\check{{g}}(0)+\omega(0). \label{7sep21b}
\end{align}
Set $G(x)={g}(x)\sum_{\mu\in(-2,2)^d} e^{2\pi i \mu\cdot x}$. Then, it is clear that $\hat{G}(x)=\sum_{\mu\in\Z^d\setminus(-2,2)^d}\check{{g}}(-x +\mu)$ is smooth in 
$\R^d\setminus(-2,2)^d$.
 Further,  $\phi(x):=\sum_{\mu\in \Z^d\setminus(-2,2)^d}\hat{{g}}(x -\mu)-\check{{g}}(0)+\omega(0)$ is smooth in $(-2,2)^d$. To see this take  $x\in (-2,2)^d$. Then $x -\mu\neq 0$ for all $\mu\in \Z^d\setminus(-2,2)^d$ and so, each function  $\hat{{g}}(\cdot -\mu)$ is smooth in $(-2,2)^d$ when   $\mu\in \Z^d\setminus(-2,2)^d$. Moreover, since each partial derivative of $\hat{g}$ decays fast, using the Lebesgue dominated convergence theorem permit us  to derivate term by term the series of $\phi$.

Further, from \cref{7sep21b} we  make the decomposition 
\bel{29may21c}
p\F \omega\F^* \overline{p}=p\F G\F^* \overline{p}+p \Phi\overline{p},
\ee
where  $ \Phi$ is the operator in $L^2(\blacksquare)$  with integral kernel $\phi(\xi-\eta)$. 
 Notice that the non-zero eigenvalues of  $p\F G \F^* \overline{p}$ are the same as those of the operator $\mathds{1}_\blacksquare p\F G\F^* \overline{p}\mathds{1}_\blacksquare$ in $L^2(\R^d)$. Then, by \cref{29may21b,29may21c} 
\begin{align*}     ||p\F|v|^{1/2}||_{2d/\gamma,w}\leq& C ||p\F|w|^{1/2}||_{2d/\gamma,w}\\ =& C ||p\F|w| \F^*\overline{p}||_{d/\gamma,w} ^{1/2} \\
\leq & C  \left(||\mathds{1}_\blacksquare p\F G \F^*\overline{p}\mathds{1}_\blacksquare||_{d/\gamma,w}+||p\Phi\overline{p}||_{d/\gamma,w}  \right)^{1/2}\\
\leq &C \left(||\mathds{1}_\blacksquare p\F|G|^{1/2}||_{2d/\gamma,w}+||p\Phi\overline{p}||^{1/2}_{d/\gamma,w}  \right) \ .\end{align*}

The following definition is taken from \cite[Subsection 5.6]{MR1306510} and we recall it for convenience of the reader.
Let $\phi\in L_{\rm loc}^2(\R^d)$ and for $\mu\in\Z^d$ define $a_\phi(\mu)=(\int_{\blacksquare+\mu} |\phi(x)|^2 dx)^{1/2}$. For $\alpha>0$ we introduce the spaces $l_{\alpha,w}(\Z^d,L^2(\blacksquare))$ as the set of functions $ \phi\in L_{\rm loc}^2(\R^d)$ that satisfies
$$
   \#\{\mu\in\Z^d:|a_\phi(\mu)|>t\}=O(t^{-\alpha}) \ .
$$

We want to  show that $G^{1/2} \in l_{2d/\gamma,w} (\Z^d,L^2(\blacksquare))$.  
 For this it is enough to notice that the sequence $(\int_{\blacksquare+\mu} \langle x\rangle^{-\gamma}dx)^{1/2}\leq C \langle \mu\rangle^{-\gamma/2}$ and \bel{25aug22}\#\{\mu:\langle \mu\rangle^{-\gamma/2}>t\}\asymp t^{-2d/\gamma}\ .\ee

  Then, since $2d/\gamma<2$,  by \cite[Subsection 5.7]{MR1306510}  
\bel{29may21d}
||\mathds{1}_\blacksquare p\F |G|^{1/2}||_{2d/\gamma,w}\leq C \Big|\Big||G|^{1/2}\Big|\Big|_{l_{2d/\gamma,w}}\Big|\Big|\mathds{1}_\blacksquare p\Big|\Big|_{L^{2d/\gamma}(\R^d)}\ .
\ee

On the other side, since $\phi$ is smooth in $B_0(\sqrt{d})$, 
it is in any Besov Space $B_{p,\infty}^\alpha((-1,1)^d)$ (notice that $(-1,1)^d=\blacksquare-\blacksquare$).
By \cite[Theorem 6, page 274]{BS87} or \cite[Theorem 6.1]{BS77}
\bel{29may21e}
||p\Phi \overline{p}||_{d/\gamma,w}\leq C ||p||^2_{L^2(\blacksquare)}.
\ee
Since $2d/\gamma<2$, we can conclude from \cref{29may21,29may21c,29may21d,29may21e,wti} that

\begin{align*}|p\F v\F^* q||_{d/\gamma,w} &\leq C \left(||p||_{L^2(\blacksquare)}+|| p||_{L^{2d/\gamma}(\blacksquare)}\right)\left(||q||_{L^2(\blacksquare)}+|| q||_{L^{2d/\gamma}(\blacksquare)}\right)\\
&\leq C ||p||_{L^2(\blacksquare)}||q||_{L^2(\blacksquare)}
\end{align*}
\end{proof}

\Cref{Cwikel-Bir-Sol} enable us to adapt the proof of \cite[Theorem 1]{Birman-Solomyak-70} to our case. 

\begin{Lemma}\label{L6_Bir-Sol}
Let $X,Y\subset\blacksquare$ such that they do not have interior points in common. Then, if $v\in S_\rho^\gamma(\Z^d)$
$$n_*\Big(s; \mathds{1}_X\F v\F^* \mathds{1}_Y\Big)=o(s^{-d/\gamma}),\quad s\downarrow 0.$$
\end{Lemma}
\begin{proof}  By \cite[Theorem 4.3.6]{RT09}  the kernel of the operator  $ \F v\F^*$ is smooth in $\T^d$  when $x\neq y$.
Suppose first that ${\rm dist}(X,Y)>0$ and consider smooth functions $\nu_X, \nu_Y$ such that $\nu_X=1$ in a neighborhood $O_X$ of $X$,  $\nu_X=0$  in a neighborhood $O_Y$ of $Y$ and ${\rm dist}(O_X,O_Y)>0$. Similarly for $\nu_Y$. Then  $=\mathds{1}_X\F v\F^* \mathds{1}_Y=\mathds{1}_X\nu_X\F v\F^* \nu_Y \mathds{1}_Y$ and the kernel of $\nu_X\F v\F^* \nu_Y $  is smooth in $\T^2$.  Then, invoking again \cite[Theorem 6, page 274]{BS87}, for any $r>1$ 
$$
n_*(s;\nu_X\F v\F^* \nu_Y)=o(s^{-r}),\quad s\downarrow 0.
$$
Suppose now that ${\rm dist}(X,Y)=0$, fix  $\epsilon>0$  and set $X_\epsilon:=\{x\in X:{\rm dist}(x,Y)>\epsilon \}$. Write $\mathds{1}_{X}\F v\F^* \mathds{1}_Y=\mathds{1}_{X_\epsilon}\F v\F^* \mathds{1}_Y+\mathds{1}_{X\setminus X_\epsilon}\F v\F^* \mathds{1}_Y.$

By \cref{Cwikel-Bir-Sol} 
 \bel{2sep21}s^{d/\gamma}n_*(s;\mathds{1}_{X\setminus X_\epsilon}\F v\F^* \mathds{1}_Y)\leq C  ||\mathds{1}_{X\setminus X_\epsilon}||_{L^2(\blacksquare)}^{\gamma/d}=o(\epsilon),\ee
uniformly for  $s>0$.  Further, by  the first part of this proof, since $\gamma>d$, 
 \bel{31may21e}
n_*(s;\mathds{1}_{X_\epsilon}\F v\F^* \mathds{1}_Y)=o(s^{-d/\gamma}),\quad s\downarrow 0.
\ee
 Putting \cref{kyfan1}  together   with  \cref{31may21e,2sep21} we conclude the proof.
\end{proof}

\begin{Lemma}\label{partition}
Let $\{\square_j\}_{j=1}^l$ be a partition of $\blacksquare$ into finite cubes of equal size. Let $\Psi_{\rm{diag}}:L^2(\blacksquare)\to L^2(\blacksquare)$ be the integral operator with kernel 
$$\sum_{j=1}^l\mathds{1}_{\square_j}(x)\Big(\sum_{k=1}^N|B_{k,j}|^2 \widehat{v}_k(x-y) \Big)\mathds{1}_{\square_j}(y),$$
where $ B_{k,j}\in\C$. Then, if $v_k$ satisfy the conditions of \cref{Th_raikov} we have 
$$n_\pm(\lambda;\Psi_{\rm{diag}})=C_{\rm{diag\pm}}\lambda^{-d/\gamma}(1+o(1)),\quad \lambda\downarrow 0,$$
with $C_{\rm{diag}\pm}=c_\pm\sum_{j=1}^l|\square_j|\Big(\sum_{k=1}^N|B_{k,j}|^2 \Gamma_k \Big)^{d/\gamma}$ and $|\square_j|$ is  the Lebesgue measure of $\square_j$.
\end{Lemma}
\begin{proof} To simplify the notation assume $N=1$.
Evidently $\Psi_{\rm{diag}}=\bigoplus_{j}^l|B_{1,j}|^2\Psi_{j, \rm{diag}}$, where the integral  kernel of $\Psi_{j, \rm{diag}}$ is $\mathds{1}_{\square_j}(x) \widehat{v}_1(x-y)\mathds{1}_{\square_j}(y)$. Therefore  
$$n_\pm(\lambda;\Psi_{\rm{diag}})=\sum_j^ln_\pm(\lambda;|B_{1,j}|^2\Psi_{j, \rm{diag}}).$$
Now, since the  kernel 
$\widehat{v}_1(x-y)=\sum_{j=1}^l\mathds{1}_{\square_j}(x)\widehat{v}_1(x-y)\mathds{1}_{\square_j}(y)+\sum_{j\neq i}^m \mathds{1}_{\square_j}(x)\widehat{v}_1(x-y)\mathds{1}_{\square_i}(y)$, 
we have that 
$$\mathcal{F}^*v_1\mathcal{F}=\bigoplus_j \Psi_{j,\rm{diag}}+R_{l},$$
where $\sum_{j\neq i}^m \mathds{1}_{\square_j}(x)\widehat{v}_1(x-y)\mathds{1}_{\square_i}(y)$ is the kernel of $R_l$.

By \cref{weyl1,L6_Bir-Sol}, for any $\delta\in(0,1)$
\begin{align*}n_\pm(\lambda;\mathcal{F}^*v_1\mathcal{F})
&\leq n_\pm\Big(\lambda(1-\delta);\bigoplus_j \Psi_{j,\rm{diag}}\Big)+o(\lambda^{-d/\gamma})
\end{align*}
and 
\begin{align*}
 n_\pm(\lambda;\mathcal{F}^*v_1\mathcal{F})&\geq n_\pm\Big(\lambda(1+\delta);\bigoplus_j \Psi_{j,\rm{diag}}\Big)+o(\lambda^{-d/\gamma}).
\end{align*}

We can take the square   $\square_0=[0,1/m)^d$, for some $m\in\Z_0$, and assume that each $\square_j$ is just  a translation of $\square_0$. Notice that in this case $l=m^d$ and $ |\square_j|=1/m^d$. Moreover,  each  operator $\Psi_{j, \rm{diag}}$  is unitary equivalent with each other,  
so the two previous inequalities imply that for any $1\leq j\leq m^d$ and $\delta\in(0,1)$
\bel{7sep21}
n(\lambda(1+\delta);v_1)+o(\lambda^{-d/\gamma})\leq m^d n_\pm(\lambda;\Psi_{j,\rm{diag}})\leq n(\lambda(1-\delta);v_1)+o(\lambda^{-d/\gamma}).
\ee

By \cref{3jun21} for all $\epsilon>0$ there exists $M$ such that $|\mu|>M$ implies 
$$|v_1(\mu)-\Gamma_1v_0(\mu)|<\epsilon |v_0(\mu)|.$$
Then, for any $\delta\in(0,1)$
\begin{align*}&n_+\left(\lambda(1+\delta);v_0^\pm(\mu) \left(\Gamma_1|B_{1,j}|^2 -\epsilon|B_{1,j}|^2\right)\right)+o(\lambda^{-d/\gamma})\\\leq &|\square_j|^{-1} n_\pm\left(\lambda; |B_{1,j}|^2\Psi_{j, \rm{diag}}\right)\\\leq &n_+\left(\lambda(1-\delta);v_0^\pm(\mu) \left(\Gamma_1|B_{1,j}|^2 +\epsilon|B_{1,j}|^2\right)\right)+o(\lambda^{-d/\gamma}).\end{align*}
Here $v_0^\pm$ are the positive and negative parts of $v_0$. 

Now, multiplying by $\lambda^{-d/\gamma}$and taking into account \cref{2jun21}, the limit when $\lambda$ goes to zero is
\begin{align*}c_\pm\left((|\Gamma_k| -\epsilon)|B_{1,j}|^2\right)^{d/\gamma} &\leq |\square_j|^{-1} \lim_{\lambda\downarrow0}\lambda^{d/\gamma}n_\pm\left(\lambda;|B_{1,j}|^2\Psi_{j, \rm{diag}}\right)\\&\leq c_\pm \left((|\Gamma_k| +\epsilon)|B_{1,j}|^2\right)^{d/\gamma} \end{align*}
Taking the limit when $\epsilon\downarrow0$ we finish the proof. 
\end{proof}

\medskip
{\bf Proof of \cref{Th_raikov}} \begin{proof}
Let $\epsilon>0$ and for any $1\leq k \leq N$ take a step function  in $\blacksquare$ of the form 
$$B_{k,\epsilon}(x)= \sum_{j=1}^lB_{k,\epsilon,j}\mathds{1}_{\square_{\epsilon,j}}(x),$$%

$||B-B_{k,\epsilon}||_{L^2(\blacksquare)}^2<\epsilon$, and where $\{\square_{\epsilon,j}\}$ are cubes chosen as in the previous lemma. Set $\Psi_\epsilon$ as the operator with integral kernel $ \sum_{k=1}^N B_{k,\epsilon}(x)\widehat{v_k}(x-y)\overline{B_{k,\epsilon}(y)}$. 
Using \cref{Cwikel-Bir-Sol} we obtain 
\bel{28may21}
||\Psi-\Psi_\epsilon||_{\frac{d}\gamma,w}\leq C \epsilon. 
\ee
Now, let $\Psi_{\epsilon,\rm{diag}}$ be the operator with integral kernel 
$$\sum_{j=1}^l\sum_{k=1}^M\mathds{1}_{\square_{\epsilon,j}}(x)\Big(|B_{k,\epsilon,j}|^2 \widehat{v}_k(x-y)
\Big)\mathds{1}_{\square_{\epsilon,j}}(y),$$
then the difference $\Psi_\epsilon-\Psi_{\epsilon,\rm{diag}}$ is the operator with kernel 
$$\sum_{i\neq j}^l\sum_{k=1}^N\mathds{1}_{\square_{\epsilon,i}}(x)\Big(B_{k,\epsilon,i} \widehat{v}_k(x-y)\overline{B_{k,\epsilon,j}}
\Big)\mathds{1}_{\square_{\epsilon,j}}(y).$$
Thus, applying \cref{L6_Bir-Sol} 
$$
\lim_{s\downarrow 0}s^{d/\gamma}n_*\Big(s;\Psi_\epsilon-\Psi_{\epsilon,\rm{diag}}\Big)=0.
$$

Using the last inequality together with \cref{kyfan1}  gives that for any $\delta>0$
$$n_\pm\Big((1+\delta)\lambda;\Psi_{\epsilon,\rm{diag}}\Big)+o(\lambda^{d/\gamma})\leq n_\pm(\lambda;\Psi)\leq n_\pm\Big((1-\delta)\lambda;\Psi_{\epsilon,\rm{diag}}\Big)+o(\lambda^{d/\gamma}), \quad \lambda\downarrow0.$$

Finally, use   \cref{partition} \begin{align*}\lim_{\lambda\downarrow0} \lambda^{d/\gamma} n_\pm(\lambda;\Psi)&= \lim_{\epsilon\downarrow0}c_\pm\sum_{j=1}^l|\square_{\epsilon,j}|\Big(\sum_{k=1}^N|B_{k,\epsilon,j}|^2 \Gamma_k \Big)^{d/\gamma}\\
&=c_\pm\int_\blacksquare \d\xi\Big(\sum_{k=1}^N\Gamma_k|B_k(\xi)|^2\Big)^{d/\gamma}
\end{align*}
\end{proof}

\section{Proofs of the main results for  parabolic thresholds and Dirac point cases}\label{proofs}

\subsection{Proof for the bounded part of the SSF}
The boundedness of the SSF near the hyperbolic thresholds was proved  in \cref{sec:hyper}. Further, for the elliptic thresholds we will only make explicit computations for $m$, the case of $\pm M_m$ being analogous.
Set

$$B:=
\U V\U^*\begin{pmatrix}
0&0&0\\[.5em]
0&-1 &0  \\[.5em]
0&0&-1\end{pmatrix}\ ,
$$  and $
\mathcal{Q}(\lambda):=B
+{\mathcal{S}}_{3}(\lambda)$.
Then, starting from \cref{lepushbis}, using \cref{21jan20,21jan20_a,15oct20,2} and \cref{lemma:pvalpha0}, together with the Weyl inequalities \cref{weyl1} and the Chebyshev-type estimate \cref{cheby} we obtain 
\begin{align}\pm  n_\mp\left( (1 \pm \epsilon);\frac{\mathcal{Q}(\lambda)}{\lambda+m}\right)+O(1)&\leq
\xib(\lambda;H_\pm,H_0)\label{9march21}\\
&\leq\pm  n_\mp\left( (1 \mp \epsilon);\frac{\mathcal{Q}(\lambda)}{\lambda+m}\right)+O(1), \label{9march21b}
\end{align}
for $|\lambda |\to m$.

The following fomula  is valid  for   $\lambda\leq 0$ when $m>0$
\begin{align*}
\mathcal{Q}(\lambda)=&\U G\U^*\frac{1}{r+m^2-\lambda^2}\begin{pmatrix}
0&0&0\\[.5em]
0&-|b|^2+\lambda^2-m^2 &\overline{a}b  \\[.5em]
0&\overline{b}a&-|a|^2+\lambda^2-m^2\end{pmatrix}\U G\U^*\\
=&  \U G\U^*\frac{1}{r+m^2-\lambda^2}\left(\begin{pmatrix}
0&0&0\\[.5em]
0&\lambda^2-m^2 &0  \\[.5em]
0&0&\lambda^2-m^2\end{pmatrix}+\begin{pmatrix}
0&0&0\\[.5em]
0&-|b|^2&\overline{a}b  \\[.5em]
0&\overline{b}a&-|a|^2\end{pmatrix}\right)\U G\U^*\ .
\end{align*}
Moreover, it is easy to see that 
\bel{9march21a}
\begin{pmatrix}
-|b|^2 &\overline{a}b  \\[.5em]
\overline{b}a&-|a|^2\end{pmatrix}=-\begin{pmatrix}
b &\overline{a} \\[.5em]
-a&\overline{b}\end{pmatrix}\begin{pmatrix}
1 &0  \\[.5em]
0&0\end{pmatrix}\begin{pmatrix}
\overline{b}  &-\overline{a} \\[.5em]
a&b\end{pmatrix}\leq0.
\ee
Therefore, for $0\geq\lambda>-m$, the operator $\mathcal{Q}(\lambda)$ is nonpositive. Thus, \cref{9march21} immediately implies the second assertion in \cref{eq:mainasym} when $m>0$.

Next, let $k$ be $-m,0$ or $m$, and let  us set 
$$l_k(\lambda)=\left\{\begin{array}{cl} |\ln(|\lambda+m|)|& \text{if}\,\, k=-m\\
1& \text{if}\,\, k=0, m\ .
\end{array}\right.   $$

Thus, from \cref{le:S10menosalpha} and  \cref{9march21,9march21b,weyl1,cheby}
\begin{align}\pm  n_\mp\left( (1 \pm \epsilon);\frac{\mathcal{Q}(k)}{\lambda+m}\right)+O(l_k(\lambda))&\leq
\xib(\lambda;H_\pm,H_0)\label{9march21c}\\
&\leq\pm  n_\mp\left( (1 \mp \epsilon);\frac{\mathcal{Q}(k)}{\lambda+m}\right)+O(l_k(\lambda)), \label{9march21d}
\end{align}
for $\lambda \to \pm m$. Now, if $m>0$, by a change of variable we obtain that
\begin{equation}\label{eq:S_myS_0}
{\mathcal{S}}_3(-m)={\mathcal{S}}_3(0) =\int_{0}^{M_0^2} \frac{\d\rho}{\rho} \int_{r^{-1}(\rho)}\d\gamma\frac{\mathring{T}_0(\gamma)}{|\nabla r (\gamma)|}= \int_{\T^2}\d\xi\frac{T_0(\xi)}{r(\xi)}\ ,
\end{equation}
and so 
\begin{equation}\label{eq:s0efectivo}
\mathcal{Q}(-m)=\mathcal{Q}(0)=\U G\U^* \frac1r\begin{pmatrix}
0&0&0\\[.5em]
0&-|b|^2 &\overline{a}b  \\[.5em]
0&\overline{b}a&-|a|^2\end{pmatrix}\U G\U^*\ .
\end{equation}
We already know that this  operator is non-positive, then  the second statement in \cref{eq:mainasym}  for $m=0$ and the first statement in \cref{eq:mainasym2} follow from \cref{9march21d}.

\subsection{Proof for the unbounded part of the SSF}
Now we turn to  the other two statements of \cref{mainTh}. Take into account \cref{9march21c,9march21d,eq:S_myS_0,eq:s0efectivo}. By an abuse of notation, we write $ 
\mathcal{Q}(0)=\U G\U^*\displaystyle{\frac{1}{r}}\begin{pmatrix}
-|b|^2 &\overline{a}b  \\[.5em]
\overline{b}a&-|a|^2\end{pmatrix}\U G\U^*. 
$  Then, if we define 
$${q}=\U G\U^*\frac{1}{\sqrt{r}}\begin{pmatrix}
b &\overline{a} \\[.5em]
-a&\overline{b}\end{pmatrix}\begin{pmatrix}
1 &0  \\[.5em]
0&0\end{pmatrix},$$
we have that $\mathcal{Q}(0)=-qq^*$, and  for all $s>0$ 
$$n_\mp(s;\mathcal{Q}(0))=n_\pm(s;q^*q).$$
In particular we have
\begin{align}\pm  n_\pm\left( (1 \pm \epsilon);\frac{q^*q}{\lambda+m}\right)+O(|\ln(|\lambda+m|)|)&\leq
\xib(\lambda;H_\pm,H_0)\label{9march21cbis}\\
&\leq\pm  n_\pm\left( (1 \mp \epsilon);\frac{q^*q}{\lambda+m}\right)+O(|\ln(|\lambda+m|)|), \label{9march21dbis}
\end{align}
for $\pm\lambda \downarrow \mp m$.

Now, set 
$V_{eff}:\Z^2\to M_n(\C^2)$ to be
$$V_{eff}(\mu):=\begin{pmatrix}v_2(\mu)& 0\\[1em]
0&v_3(\mu)\end{pmatrix}\ .$$
The  positive eigenvalues of the operator $q^*q=\A^*\F V_{eff}\F^*\A$ coincide with the ones of the integral operator in $L^2(\T^2;\C)$ 
with  kernel 
$$
\frac{\overline{b(\xi)}}{\sqrt{r(\xi)}}\widehat{v_2}(\xi-\eta)\frac{{b(\eta)}}{\sqrt{r(\eta)}}+\frac{\overline{a(\xi)}}{\sqrt{r(\xi)}}\widehat{v_3}(\xi-\eta)\frac{{a(\eta)}}{\sqrt{r(\eta)}} \ .
$$
This kernel is of the form \cref{eq:kernelefectivo}, then it defines an  operators $\Psi$ as in \cref{Raikov}.  To finish the proof it is enough to use \cref{Th_raikov}  in the case $d=2$. To this end, it only remains to show that $v_2$ and $v_3$ satisfy \emph{Condition 1}.
First, \cref{3jun21}  correspond to condition  \cref{14jan21} by  considering $v_0(\mu)=\langle\mu\rangle^{-\gamma}$.  Second, $\langle\mu\rangle^{-\gamma}$  satisfies \cref{2jun21} as is shown by \cref{25aug22}. Finally,  that $v_2$, $v_3$ are in $S_\rho^\gamma(\Z^2)$ is just condition \cref{31may21}.

{\bf Acknowledgments:} P. Miranda was supported by the Chilean Fondecyt Grant $1201857$. D. Parra was supported by the Chilean Fondecyt Grant $3210686$. They would also like to deeply thank Lilia Simeonova who made available the notes of the deceased G. Raikov, upon which the \cref{Raikov} was developed. The main ideas were proposed by G. Raikov but any potential error is solely attributed to the first two authors.  We would also like to thank the anonymous referee for the careful reading and useful suggestions. 

\printbibliography
\end{document}